\documentclass[final]{siamltex}
\usepackage{amsmath}
\usepackage{amsfonts}
\usepackage{amssymb}
\usepackage{graphicx}

\usepackage{url}
\usepackage{xcolor}
\usepackage{fullpage}
\usepackage{bm}
\usepackage{mathrsfs}
\usepackage{nccmath}
\DeclareMathAlphabet{\itbf}{OML}{cmm}{b}{it}

\newcommand{\CC}{\mathbb{C}}
\newcommand{\RR}{\mathbb{R}}

\newcommand{\om}{\omega}

\newcommand{\ep}{\varepsilon}

\newcommand{\wh}{\widehat}

\newcommand{\cS}{\mathcal{S}}
\newcommand{\cC}{\mathcal{C}}
\newcommand{\cI}{\mathscr{I}}
\newcommand{\cd}{\mathfrak{D}}
\newcommand{\cB}{\mathfrak{B}}
\newcommand{\cL}{\mathfrak{L}}

\newcommand{\brho}{\bm{\rho}}
\newcommand{\btau}{\bm{\tau}}
\newcommand{\bX}{{\itbf X}}
\newcommand{\bU}{{\itbf U}}
\newcommand{\bG}{{\itbf G}}
\newcommand{\bV}{{\itbf V}}
\newcommand{\bD}{{\itbf D}}
\newcommand{\bW}{{\itbf W}}
\newcommand{\bE}{{\itbf E}}
\newcommand{\bZ}{\itbf{Z}}
\newcommand{\rc}{r_{_{\hspace{-0.02in}\cC}}}

\newcommand{\bcW}{\boldsymbol{\mathcal W}}
\newcommand{\bu}{{\itbf u}}
\newcommand{\bv}{{\itbf v}}
\newcommand{\by}{{\itbf y}}
\newcommand{\br}{{\itbf r}}
\newcommand{\bx}{{\itbf x}}
\newcommand{\bmm}{{\itbf m}}
\newcommand{\be}{{\itbf e}}
\newcommand{\bd}{{\itbf d}}
\newcommand{\bz}{{\itbf z}}

\newcommand{\bnu}{\boldsymbol{\nu}}
\newcommand{\bga}{\boldsymbol{\gamma}}
\newcommand{\bg}{{\itbf g}}
\newcommand{\bvr}{{\itbf v}_{\rightarrow}}
\newcommand{\bxi}{\boldsymbol{\mathfrak{X}}}

\newcommand{\bEr}{\boldsymbol{\mathcal{E}}}

\newcommand{\bR}{\itbf{R}}

\newcommand{\obr}{\overline{\br}}
\newcommand{\oby}{\overline{\by}}
\newcommand{\oL}{\overline L}
\newcommand{\oom}{\overline{\om}}
\newcommand{\ok}{\overline{k}}
\newcommand{\ot}{\overline{t}}
\newcommand{\ola}{\overline{\lambda}}

\newcommand{\lin}{\left <}
\newcommand{\rin}{\right >}
\renewcommand{\hat}{\widehat}

\newenvironment{mpmatrix}{\begin{medsize}\begin{pmatrix}}%
    {\end{pmatrix}\end{medsize}}%

\begin{document}

\title{A Multiple Measurement Vector approach to Synthetic
  Aperture Radar Imaging} \author{Liliana Borcea \and Ilker Kocyigit
  \footnotemark[1]} \renewcommand{\thefootnote}{\fnsymbol{footnote}}
\footnotetext[1]{Department of Mathematics, University of Michigan,
  Ann Arbor, MI 48109-1043. \\\hspace{0.3in}Email: borcea@umich.edu \&
  ilkerk@umich.edu} \maketitle \date{today}
\begin{abstract}
We study a multiple measurement vector (MMV) approach to synthetic
  aperture radar (SAR) imaging of scenes with direction dependent
  reflectivity and with polarization diverse measurements. The
  unknown reflectivity is represented by a matrix with row support
  corresponding to the location of the scatterers in the scene, and
  columns corresponding to measurements gathered from different
  sub-apertures, or different polarization of the waves. The MMV
  methodology is used to estimate the reflectivity matrix by inverting
  in an appropriate sense the linear system of equations that models
  the SAR data.  We introduce a resolution analysis of  imaging with
  MMV, which takes into account the sparsity of the imaging scene, the
  separation of the scatterers and the diversity of the
  measurements. The results of the analysis are illustrated with some 
  numerical simulations.
\end{abstract}


\section{{Introduction}}
\label{sect:intro}
Sparsity promoting optimization
\cite{donoho1989uncertainty,donoho1992signal,donoho2006compressed,
  bruckstein2009sparse,candes2005decoding,candes2006robust,candes2006near}
is an important methodology for imaging applications where scenes that
are sparse  in some representation can be reconstructed
with high resolution.  There is a large body of literature on this topic in
synthetic aperture radar imaging
\cite{baraniuk2007compressive,potter2010sparsity,fannjiang2013compressive},
sensor array imaging
\cite{chai2013robust,chai2014imaging,borceaKocyigit}, medical imaging
\cite{lustig2007sparse}, astronomy \cite{bobin2008compressed},
geophysics \cite{santosa1986linear},  and so on.

We are interested in the application of synthetic aperture radar (SAR)
imaging, where a transmit-receive antenna on a moving platform
probes an imaging scene with waves and records the scattered returns
\cite{curlander1991synthetic,cheney2009fundamentals}. This is a
particular inverse problem for the wave equation, where the waves
propagate through a homogeneous medium, back and forth between the
platform and the imaging scene, and the unknown is modeled as a
two-dimensional reflectivity function of location on a known imaging
surface. Most SAR imaging is based on a linear
model of the data, where the unknown reflectivity is represented by a
collection of independent point scatterers
\cite{cheney2009fundamentals}. The image is then formed by inverting
approximately this linear relation, using filtered backprojection or
matched filtering \cite{cheney2009fundamentals}, also known as
Kirchhoff migration \cite{Biondi}.  Such imaging is
popular because it is robust to noise, it is simple and works well
when the linear model is a good approximation of the data. However,
the resolution is limited by the extent of the  aperture, the
frequency and the bandwidth of the probing signals emitted by the
moving platform
\cite{curlander1991synthetic,cheney2009fundamentals}. The promise of
sparsity promoting optimization is that these resolution limits can be
overcome when the unknown reflectivity has sparse support
\cite{baraniuk2007compressive,potter2010sparsity,
  fannjiang2013compressive}.

The modeling of the reflectivity as a collection of points that
scatter the waves isotropically may lead to image artifacts.  It is
known that even if the scatterers are small, so that their support may
be represented by a point and the single scattering approximation
(i.e., the linear data model) can be used, their reflectivity may
depend on the frequency and the direction of illumination
\cite[Chapters 3, 5]{ammari2013mathematical}. Moreover, the scatterers
have an effective polarization tensor that describes their response to
different polarizations of the probing electromagnetic waves
\cite{ammari2013mathematical,ammari2007music}. Thus, the reflectivity function depends
on more variables than assumed in
conventional SAR, and the
resulting images may be worse than expected. For example, a scatterer
that reflects only within a narrow cone of incident angles cannot be
sensed over most of the synthetic aperture, so its reconstruction with
filtered backprojection will have low resolution.  Direct application
of sparse optimization methods does not give good results either,
because of the large systematic error in the linear data model that
assumes a scalar, constant reflectivity over the entire aperture.

SAR imaging of frequency-dependent reflectivities has been studied in
\cite{cheney2013imaging,sotirelis2012study,sotirelis2013frequency},
using either Doppler effects, or data segmentation over frequency
sub-bands. Data segmentation is a natural idea for imaging both
frequency and direction dependent reflectivities that are regular enough
so that they can be approximated as piecewise constant functions over
properly chosen frequency sub-bands and cones of angles of incidence
(i.e., sub-apertures). Images can be obtained separately from each
data set, but the question is how to fuse the information to
achieve better resolution.  The study in \cite{borcea2016synthetic}
uses the multiple measurement vector (MMV) methodology
\cite{malioutov2005, Chen2006MMV, VanDenBerg2010MMV}, also known as
simultaneously sparse approximation
\cite{TroppSimultaneous1,TroppSimultaneous2}, for this purpose. The
MMV framework fits here because the reflectivity is supported at the
same locations in the imaging scene, for each data set.  In the discrete setting, this
means that the unknown is represented by a matrix ${\itbf X}$ with row
support corresponding to the pixels in the image that contain
scatterers, and with columns corresponding to the different values of
the reflectivity for each frequency band, sub-aperture and polarization.

The goal of this paper is two-fold: First, we  introduce a novel resolution theory 
of imaging with MMV,  that applies to a general linear system. We do not pursue the usual 
question of exact recovery of the unknown matrix $\bX$, which requires stringent 
assumptions on the imaging scene that are unlikely to hold in practice. Instead, we 
estimate the neighborhood of the  row  support  of $\bX$ that contains the largest entries 
of the MMV reconstruction. The size of this neighborhood plays the role of resolution limit and 
we quantify its dependence on the sparsity of the imaging scene, the separation between the 
scatterers, the diversity of the data set and the noise level. The second goal of the paper is to explain how  
the  theory applies to SAR imaging. The  study \cite{borcea2016synthetic} is proof of concept 
that MMV can be used to image  direction dependent 
reflectivities from data gathered over multiple sub-apertures. However, 
it does  not provide a resolution analysis and it does not demonstrate the advantage of using  MMV 
over imaging with a single sub-aperture at a time. 
In this paper we quantify the improvement brought by the MMV approach
and assess the results of the resolution theory for  the application of  SAR imaging both direction and polarization dependent 
reflectivities.

The paper is organized as follows: We begin in section
\ref{sect:theory} with the theoretical results, stated for a general
linear system with unknown matrix ${\itbf X}$. The application of SAR 
imaging is discussed in sections  \ref{sect:applic1} and  \ref{sect:applic2}. The proofs of the
results  are in section \ref{sect:proofs}.  We end with a summary in
section \ref{sect:sum}.

\section{Theory}
\label{sect:theory}
We state here our main results on the resolution of imaging with
MMV. We begin in section \ref{sect:theory1} with a brief discussion on 
MMV, and then give the results in section \ref{sect:theory2}.

We use henceforth the following notation convention: Bold uppercase
letters, as in $\bX \in \CC^{N_\by \times N_v}$, denote matrices and
bold lowercase letters denote vectors. We also use an arrow index, as
in $\bx_{j\rightarrow}\in \CC^{1 \times N_v}$, to distinguish the rows
of $\bX$ from its column vectors denoted by $\bx_{j} \in \CC^{N_{\by}
  \times 1}$.
\subsection{Preliminaries}
\label{sect:theory1}
Consider a general linear model of a data matrix $\bD \in \CC^{N_{\br}
  \times N_v}$,
\begin{equation}
  \bG \bX = \bD,
  \label{eq:th1}
\end{equation}
where the unknown matrix $\bX \in \CC^{N_{\by} \times N_v}$ is mapped
to $\bD$ by a given sensing matrix $\bG \in \CC^{N_{\br}\times
  N_{\by}}$.  In the context of SAR imaging, $\bX$ is the unknown
reflectivity discretized\footnote{We assume that the $N_{\by}$ points
  define a fine mesh in $\Omega$, so we can neglect errors due to
  scatterer locations off the mesh.} at $N_{\by}$ points $\{\by_j\}_{1
  \le j \le N_{\by}}$ in the imaging region $\Omega$, a bounded set on
a known surface.  The matrix $\bD$ is an aggregate of $N_v$ data sets
or views, each consisting of $N_{\br}$ measurements of the wave 
at the moving radar antenna. The column 
$\bx_v$ of $\bX$ is the reflectivity for the $v$-th
view, and 
the sensing matrix $\bG$ is the discretization of the
kernel of the integral operator that defines the single
scattering  approximation of the wave,
as described in section \ref{sect:applic1}.  

Denote by $\cS \subset \{1, \ldots, N_{\by}\}$ the set of indexes of
the nonzero rows of $\bX$, and suppose that its cardinality $|\cS|$ is
small with respect to $ N_{\by}$. We call $\cS$ the row support of
$\bX$ and let $ \Omega_\cS = \{ \by_q, ~ q \in \cS\} $ be the set of
associated locations in $\Omega$.

When $N_v = 1$, the linear model \eqref{eq:th1} corresponds to the
single measurement vector (SMV) problem,
\begin{equation}
  \bG \bx = \bd,
  \label{eq:th2}
\end{equation}
with unknown vector $\bx \in \CC^{N_{\by}\times 1}$ and data vector
$\bd \in \CC^{N_{\br}\times 1}$, where we dropped the column index $1$. 
This problem has been studied extensively in the
context of compressed sensing
\cite{donoho1989uncertainty,donoho1992signal,donoho2006compressed,
  bruckstein2009sparse,candes2005decoding,candes2006robust,candes2006near,
  fannjiang2010compressed} for the undetermined case $N_{\br} \ll
N_{\by}$. In particular, it is known \cite[Corollary
  1]{donoho2003optimally} that if
\begin{equation}
  \|\bx\|_0 = |\cS| < { \mbox{spark}(\bG)}/{2},
  \label{eq:th4}
\end{equation}
where $\mbox{spark}(\bG)$ is the smallest
  number of linearly dependent columns of $\bG$, then 
\eqref{eq:th2} has a unique solution satisfying \eqref{eq:th4}, given
by the minimizer of the combinatorial  optimization problem
\begin{equation}
  \mbox{minimize} ~ \|\bz\|_0 ~ ~\mbox{subject to } ~
  \bG \bz = \bd.
  \label{eq:th3}
\end{equation}
The norm $\| \bz\|_0$ equals the number of nonzero entries in $\bz$.

This result is generalized in \cite[Theorem 2.4]{Chen2006MMV} to the
MMV problem \eqref{eq:th1} for $N_v > 1$. It states that when the
number of nonzero rows in $\bX$, denoted by $\|\bX\|_0$, satisfies
\begin{equation}
  \|\bX\|_0 < \big[ \mbox{spark}(\bG)+ \mbox{rank}({\bD})-1\big]/{2},
  \label{eq:th5}
\end{equation}
the linear system \eqref{eq:th1} has a unique solution satisfying
\eqref{eq:th5}, given by the minimizer of 
\begin{equation}
  \mbox{minimize} ~ \|\bZ\|_0 ~ ~\mbox{subject to } ~
  \bG \bZ = \bD.
  \label{eq:th6}
\end{equation}
Thus, if the different data sets bring new information, so that $\bD$
has large rank, the MMV problem is uniquely solvable for less
stringent conditions on the row support of $\bX$ i.e., for less sparse
imaging scenes.

The combinatorial problems \eqref{eq:th3} and \eqref{eq:th6} are not
computationally tractable, so they are replaced by convex
relaxations. The minimizer of the convex problem
\begin{equation}
  \mathscr{P}_1: ~ ~  \mbox{minimize} ~ \|\bz\|_1 ~ ~\mbox{subject to } ~
  \bG \bz = \bd,
  \label{eq:th7}
\end{equation}
where $\|\cdot \|_1$ is the $\ell_1$ norm, is known to give the exact
solution $\bx$ of  \eqref{eq:th2} under various
conditions satisfied by $\bG$ and $\bx$, like the null space property
\cite{Cohen09compressedsensing}, the restricted isometry property
\cite{candes2006robust}, conditions based on the mutual coherence
\cite{Donoho2006Uncertainty} and the cumulative coherence
\cite{tropp2004greed}. Relaxations of \eqref{eq:th6} of the form
\begin{equation}
   \mathscr{P}_{1,q}: ~ ~ \mbox{minimize} ~ \|\bZ\|_{1,q} ~
   ~\mbox{subject to } ~ \bG \bZ = \bD, \quad \mbox{where} ~ ~  \|\bZ\|_{1,q} = \sum_{j=1}^{N_{\by}} \|\bz_{j\rightarrow}\|_q,
  \label{eq:th8}
\end{equation}
are studied in
\cite{cotter2005sparse,malioutov2005,Chen2006MMV,Eldar2009Robust,
  TroppSimultaneous1,TroppSimultaneous2,VanDenBerg2010MMV}. Conditions of recoverability of  $\bX$  by the minimizer of 
\eqref{eq:th8} are established in \cite[Theorem 3.1]{Chen2006MMV} and
\cite[Theorem 5.1]{TroppSimultaneous1}. However, there are no
conclusive results that demonstrate the advantage of the MMV
formulation over the SMV one in the convex relaxation form, as
discussed for example in \cite[Section D]{Chen2006MMV}, \cite[Section
  5.2]{TroppSimultaneous1} and \cite[Section 3.2]{VanDenBerg2010MMV}.

These studies make no assumption on the structure of the unknown
$\bX$, except for sparsity of its row support $\cS$, and  do not address 
the case of more general imaging scenes where exact reconstructions of 
$\bX$ may not be achieved.  Our resolution theory quantifies  the error
of the reconstruction based on the separation between the points in $\Omega_\cS$, the 
correlation of the rows of $\bX$ and the noise level. We show in particular 
that if $\bX$ has uncorrelated rows,  the MMV formulation may have an advantage
over SMV.  This is relevant to SAR imaging, as explained in section \ref{sect:applic1}.

\subsection{Resolution theory}
\label{sect:theory2}
Let us consider the following modification of the linear system
\eqref{eq:th1}
\begin{equation}
  \bD_{\bW} = \bG \bX + \bW,
  \label{eq:th10}
\end{equation}
which accounts for data $\bD_{\bW}\in \CC^{N_{\br} \times N_v}$
contaminated by the noise matrix $\bW \in \CC^{N_{\br} \times
  N_v}$. We estimate $\bX$ by the minimizer $\bX^\ep$ of
the convex problem
\begin{equation}
  \mathscr{P}_{1,2}^\ep: ~~ \mbox{minimize} ~ \|\bZ\|_{1,2} ~
   ~\mbox{subject to } ~ \|\bG \bZ - \bD_{\bW}\|_{F} \le \ep,
  \label{eq:th11}
\end{equation}
where $\|\cdot \|_F$ is the Frobenius norm and $\ep$ is a chosen 
tolerance, satisfying
\begin{equation}
  \|\bW\|_{F} < \ep.
  \label{eq:th12}
\end{equation}

Our goal  is to quantify the approximation of $\bX$ by $\bX^\ep$, by taking into account 
the separation of the points in  $\Omega_\cS$ and the correlation of the rows of $\bX$.  These determine how the unknowns  interact with each other, 
as described by the  $\bX$ dependent "multiple view interaction coefficient'' $\cI_{N_v}$ 
defined in section \ref{sect:defINv}.  The smaller $\cI_{N_v}$ is, the 
 better the imaging results, as stated by the estimates in sections 
\ref{sect:suppX}--\ref{sect:orthogX}.  We also study 
in section \ref{sect:cluster} the case of clusters of points in $\Omega_\cS$, where $\cI_{N_v}$ is large
and the previous estimates  are not useful.  We introduce a new interaction 
coefficient for the cluster, which is much smaller than $\cI_{N_v}$, and show that when this 
is small, the MMV reconstruction is supported in the vicinity of $\Omega_\cS$.

\subsubsection{The multiple view interaction coefficient}
\label{sect:defINv}
The interaction between the unknowns is quantified by the $\bX$ dependent multiple view interaction coefficient defined by 
\begin{equation}
  \cI_{N_v} = \max_{1 \le j \le N_{\by}} \sup_{\bvr \in \CC^{1 \times
      N_v}} \sum_{q \in \cS \setminus\{n(j)\}} |\mu(\bg_j,\bg_q)|
  |\mu(\bvr,\bx_{q\rightarrow})|,
  \label{eq:th13}
\end{equation}
using the correlation of the columns of $\bG$,
\begin{equation}
  \mu(\bg_j,\bg_q) = \lin \bg_j,\bg_q \rin,
  \quad 1 \le j, q \le N_{\by},
    \label{eq:th14}
\end{equation}
where $\lin \bg_j,\bg_q \rin = \bg_j^\star \bg_q$ is the Hermitian
inner product, and $\star$ denotes complex conjugate and transpose. 
These columns are normalized, so that
\begin{equation}
  \label{eq:normg} 
\|\bg_j\|_2 = \lin \bg_j,\bg_j \rin^{1/2} = 1, \quad 1 \le j \le N_{\by},
\end{equation}
and we suppose that
\begin{equation}
\label{eq:th16b}
|\mu(\bg_j,\bg_q)| < 1, \quad \forall ~ j \ne q, ~ 1 \le j, q \le N_{\by}.
\end{equation}
This assumption holds in the SAR imaging application and it allows us
to quantify the distance between the points  using the
semimetric
\begin{equation}
  \cd:\{1, \ldots, N_{\by}\} \times \{1, \ldots, N_{\by}\}
  \to [0,1], \quad \cd(j,q) = 1 - |\mu(\bg_j,\bg_q)|.
  \label{eq:th16}
\end{equation}
 We will see in section \ref{sect:applic1} that
$|\mu(\bg_j,\bg_q)|$ is approximately a function of $\by_j-\by_q$,
which peaks at the origin i.e., for $\by_j = \by_q$, and decreases
monotonically in the vicinity of the peak. Thus, points at small
distance with respect to $\cd$ are also close in the Euclidian
distance.

We use the semimetric $\cd$ in definition \eqref{eq:th13} to select
the closest point to $\by_j$ in
$\Omega_\cS$, indexed by $n(j) \in \cS$.  If this point is not unique,  
we just pick one and let $n(j)$ be its index. In an abuse of notation, we also let
$\mu(\cdot, \cdot)$ be the correlation of the rows of $\bX$ with
$\bvr$, defined by
\begin{equation}
  \mu(\bvr,\bx_{q \rightarrow}) = \frac{ \lin \bvr,\bx_{q \rightarrow}
    \rin}{ \|\bvr \|_2 \|\bx_{q \rightarrow}\|_2},
    \label{eq:th15}
\end{equation}
where $\lin \bvr, \bx_{q \rightarrow} \rin = \bvr\bx_{q
  \rightarrow}^\star $ is the Hermitian inner product of row vectors
and $\|\cdot \|_2$ is the induced $\ell_2$ norm.

Note that \eqref{eq:th15} has absolute value equal to
$1$ in the SMV setting, where $N_v = 1$ and $\bvr$ and $\bx_{k
  \rightarrow}$ are complex numbers. Then, \eqref{eq:th13} reduces to
the single view interaction coefficient $\cI_{1}$ used in \cite[Section
  4]{borceaKocyigit} to quantify the quality of imaging
 with $\ell_1$ optimization. As shown in \cite{borceaKocyigit}, 
$\cI_1$ is small if the points in $\Omega_\cS$ are sufficiently far apart. Here we consider $N_v >1$,
and note that since $|\mu (\bvr,\bx_{k \rightarrow})| \le 1$, we have $\cI_{N_v} \le \cI_1$. 
In section  \ref{sect:orthogX} we show that depending on the correlation of the rows of $\bX$, 
we may have $\cI_{N_v} \ll \cI_1$. The resolution estimates below 
show an advantage of using MMV in such cases.


\subsubsection{Estimation of the support of $\bX$}
\vspace{0.05in} \label{sect:suppX}

The next theorem, proved in section \ref{sect:proof1}, shows that when
$\cI_{N_v}$  and the noise level $\ep$ are small,  
the large entries in $\bX^\ep$  are supported at points near
$\Omega_\cS$. 
\vspace{0.05in}
\begin{theorem}
\label{thm:1}
Consider the matrix $ \bW^\ep = \bG (\bX^\ep - \bX), $ defined in
terms of the unknown $\bX$ and its reconstruction $\bX^\ep$,
the minimizer of \eqref{eq:th11}.  This matrix cannot be computed 
but it is guaranteed to satisfy 
\begin{equation}
    \label{eq:th19p}
    \|\bW^\ep\|_{F} \le 2 \ep.
\end{equation}
Suppose that there exists $r \in (0,1)$ so that $2 \cI_{N_v} < r < 1,$
and define the set
\begin{equation*}
    \cB_r(\cS) = \{ 1 \le j \le N_{\by} ~ \mbox{such that} ~ \exists
    \, q \in \cS ~ \mbox{satisfying} ~ \cd(j,q) < r \},
\end{equation*}
called the $r$--vicinity of $\cS$ with respect to the semimetric
$\cd$.  If we decompose the reconstruction in two parts
\begin{equation}
    \bX^\ep = \bX^{\ep,r} + \bE^{\ep,r},
    \label{eq:th18}
\end{equation}
whith $\bX^{\ep,r}$ row supported in $\cB_r(\cS)$ and $\bE^{\ep,r}$  
row supported in  the complement $\{1, \ldots, N_{\by}\} \setminus \cB_r(\cS)$, we have
\begin{align}
    \|\bE^{\ep,r}\|_{1,2} \le \frac{2 \cI_{N_v}}{r} \|\bX^\ep\|_{1,2} +
    \frac{1}{r} \big\| \big(\bG^\star \bW^\ep\big)_{\cS \rightarrow}
    \big\|_{1,2} \le \frac{2 \cI_{N_v}}{r} \|\bX^\ep\|_{1,2} + \frac{2 \ep |\cS|}{r},
    \label{eq:th19}
\end{align}
where $\bG^\star \in \CC^{N_{\by} \times N_{\br}}$ is the Hermitian
adjoint of $\bG$ and $\big(\bG^\star \bW^\ep\big)_{\cS \rightarrow} \in
\CC^{|\cS| \times N_{v}}$ is the restriction of the matrix
$\bG^\star \bW^\ep$ to the rows indexed by the entries in $\cS$. 
\end{theorem}

\vspace{0.05in} We may think of $\bE^{\ep,r}$ as an
error in the reconstruction, because its rows are supported away
from $\cS$. The theorem says that this error is small when the multiple interaction 
coefficient and the noise level are small. The estimate of the noise effect in the second bound 
in \eqref{eq:th19} is pessimistic.  In the numerical simulations we found that 
$\big\| \big(\bG^\star \bW^\ep\big)_{\cS \rightarrow} \big\|_{1,2}$ is
typically much smaller than $2 \ep |\cS|$.

\subsubsection{Quantitative estimation of $\bX$}
\label{sect:quant}
Theorem \ref{thm:1} says that if we threshold the entries in $\bX^\ep$ at a value commensurate to the right hand side in  \eqref{eq:th19}, we obtain the approximation $\bX^{\ep,r}$ with row support $\cS^\ep \subset \cB_r(\cS)$. 
Here we quantify how well  
$\bX^{\ep,r}$ approximates $\bX$.  Because $\cS$ and $\cS^\ep$ are
different sets in general, an estimate of some norm of $\bX^{\ep,r}-\bX$ is not 
useful. Instead, we decompose $\bX^{\ep,r}$ in 
one part supported in $\cS$ that we compare with $\bX$ in Theorem \ref{thm:2}, and a residual. 

Let  $\bG_{\cS} = (\bg_j)_{j \in \cS}$ be the $N_{\br} \times
|\cS|$ matrix obtained by restricting the columns of $\bG$ to the
indexes in $\cS$. Suppose that $\bG_{\cS}$ has linearly independent
columns, as otherwise it is impossible to recover $\bX$ even with
noiseless data, and introduce its pseudoinverse
\begin{equation}
  \bG_\cS^\dagger = (\bG_\cS^\star \bG_\cS)^{-1} \bG_\cS^*.
  \label{eq:th20}
\end{equation}
Decompose $\bX^{\ep,r}$   in two parts
\begin{equation}
  \bX^{\ep,r} = \bxi^{\ep,r} + \bEr^{\ep,r},
\label{eq:th21}
\end{equation}
where $\bxi^{\ep,r}$ has row support in $\cS$ and its restriction to the
rows indexed by $\cS$ satisfies
\begin{equation}
  \bxi^{\ep,r}_{\cS \rightarrow} = \bG_\cS^\dagger \bG \bX^{\ep,r}.
  \label{eq:th22}
\end{equation}
This definition gives that 
\begin{align}
  \bG_\cS^\dagger \bG \bX^{\ep,r} = (\bG_\cS^\star \bG_\cS)^{-1}
  \bG_\cS^* \Big( \bG_\cS \bxi^{\ep,r}_{\cS \rightarrow} + \bG \bEr^{\ep,r}
  \Big) = \bxi^{\ep,r}_{\cS \rightarrow} + (\bG_\cS^\star \bG_\cS)^{-1}
  \bG_\cS^* \bG \bEr^{\ep,r},
  \label{eq:calcul}
\end{align}
so the residual $\bEr^{\ep,r}$ satisfies
\begin{equation}
  \bG_\cS^* \bG \bEr^{\ep,r} = 0.
  \label{eq:th23}
\end{equation}
That is to say, the columns of $\bG \bEr^{\ep,r}$ are orthogonal to
the range of $\bG_\cS$.  Note that $\bEr^{\ep,r}$ has row
support in $\cS \cup \cS^\ep$. If $\cS^\ep$ were the same as
$\cS$, then \eqref{eq:th23} would imply that $\bEr^{\ep,r} = 0$. Thus,
$\bEr^{\ep,r}$ is a residual that accounts for
$\bX^{\ep,r}$ not having the exact support $\cS$.

The next theorem, proved in section  \ref{sect:proof2},
shows that under the same conditions as in Theorem \ref{thm:1},  the matrix $\bxi^{\ep,r}$ is a good
approximation of the unknown $\bX$. However, $\bxi^{\ep,r}$ cannot be computed directly, so we 
need to relate it to $\bX^{\ep,r}$. To do so, we introduce an ``effective matrix''  supported in 
$\cS$, obtained by local aggregation of the rows of $\bX^{\ep,r}$. We show that $\bxi^{\ep,r}$ is close to 
to this matrix if  the single view interaction coefficient $\cI_{1}$  is small. This reveals the fact that 
while $\cI_{N_v} \ll \cI_{1}$  brings an improved support of the MMV reconstruction vs. that of  SMV, 
the quantitative estimate of $\bX$ cannot be expected to be better.

\vspace{0.05in}
\begin{theorem}
\label{thm:2}
Let $\bX^{\ep,r}$ and $\bxi^{\ep,r}$ be defined as in \eqref{eq:th18} and
\eqref{eq:th21}. Then, 
\begin{equation}
  \|\bxi^{\ep,r} - \bX\|_{1,2} 
  \le \frac{2 \cI_{N_v}}{r}
  \|\bX^\ep\|_{1,2} + \frac{6 \ep |\cS|}{r}.
  \label{eq:th27}
\end{equation}
Moreover, if the support of $\bX^{\ep,r}$ is 
decomposed in $|\cS|$ disjoint parts,  each
corresponding to a point in $\cS$, 
\begin{equation}
  \label{eq:th24}
  \cS^\ep = \bigcup_{j \in \cS} \cS^\ep_j, \quad
  \cS^\ep_j = \{ q \in \cS^\ep ~ \mbox{such that} ~ \cd(q,j)
  \le \cd(q,j'), ~ \forall ~ j' \in \cS\}, \quad  j \in
  \cS,
\end{equation}
and we define  the effective matrix $\overline{\bX^{\ep,r}}$ with row support in $\cS$ and 
 entries 
\begin{equation}
  \label{eq:th25}
  \overline{\bX^{\ep,r}_{j,v}} = \left\{ \begin{array}{ll}
    \displaystyle \sum_{l \in \cS^\ep_j} 
    \mu(\bg_j,\bg_l)\bX^{\ep,r}_{l,v}, \quad &\mbox{if} ~ j \in \cS, \\ 0,
    &\mbox{otherwise},
  \end{array} \right. \quad \mbox{for} ~ ~1 \le j \le N_{\by}, ~ 1 \le v \le
  N_v,
\end{equation}
we have the estimate 
\begin{equation}
    (1-\cI_{1})\| \bxi^{\ep,r} - \overline{\bX^{\ep,r}}\|_{1,1} \le 2
  \cI_{1}\|\bX^{\ep,r}\|_{1,1}.
    \label{eq:th26}
\end{equation}
\end{theorem}

Note that because $\mu(\bg_j,\bg_l)$  are complex valued, there may be cancellations in the 
local aggregation  \eqref{eq:th25} of the entries of $\bX^{\ep,r}$. Only if  the set 
$\cS^\ep_j$ is small, so that $\cd(j,l) \ll 1$ for $l \in \cS^\ep_j$, we have $\mu(\bg_j,\bg_l) \approx 1$
and \eqref{eq:th25} is approximately the local sum of the entries in $\bX^{\ep,r}$. 
%

\subsubsection{Matrices $\bX$ with orthogonal rows}
\label{sect:orthogX}
We now show that if the unknown matrix $\bX$ has
orthogonal rows\footnote{The results extend to nearly orthogonal rows, 
but to simplify the proof we assume orthogonality.} (i.e., uncorrelated), then the multiple view interaction coefficient
$\cI_{N_v}$ may be much smaller than the interaction coefficient
$\cI_1$. By Theorem \ref{thm:1}, this means
that the MMV approach can give improved estimates of the row support $\cS$
of $\bX$, under less stringent conditions  than in the SMV formulation.

\vspace{0.05in}
\begin{proposition}
\label{prop:1}
Suppose that  $\bX \in \CC^{N_{\by} \times N_v}$ has
row support in the set $\cS$ with cardinality $1 < |\cS| \le N_v$, and
that its nonzero rows are orthogonal. Then, the multiple view
interaction coefficient \eqref{eq:th13} is given by
\begin{equation}
  \cI_{N_v} = \max_{1 \le j \le N_{\by}} \sqrt{\sum_{q \in \cS
      \setminus \{n(j)\}} |\mu(\bg_j,\bg_q)|^2}.
  \label{eq:Prop1}
\end{equation}
\end{proposition}

\vspace{0.05in} This proposition, proved in section \ref{sect:proof3},
gives a simpler expression of $\cI_{N_v}$, that  we can compare with 
\begin{equation}
  \cI_1 = \max_{1 \le j \le N_{\by}} \sum_{q \in \cS \setminus
    \{n(j)\}} |\mu(\bg_j,\bg_q)|,
  \label{eq:O5}
\end{equation}
to understand when $\cI_{N_v} \ll \cI_1$.  For this purpose, let us
define the vector $\bga^{(j)} \in \RR^{1 \times (|\cS|-1)}$ with
entries $|\mu(\bg_j, \bg_q)|$, for $q \in \cS \setminus\{n(j) \}$, and  rewrite
\eqref{eq:Prop1} and \eqref{eq:O5} as 
\begin{equation}
  \cI_{N_v} = \max_{1 \le j \le N_{\by}} \|\bga^{(j)}\|_2, \qquad
  \cI_1 =  \max_{1 \le j \le N_{\by}} \|\bga^{(j)}\|_1,
  \label{eq:Ort1}
\end{equation}
using the $\ell_2$ and $\ell_1$ vector norms.  Suppose
that the maximizer in the definition of $\cI_{N_v}$ is at index $j =
m$. Basic vector norm inequalities give the general relation
\begin{equation*}
  \cI_{N_v} = \|\bga^{(m)}\|_2 \le \|\bga^{(m)}\|_1 \le \cI_1,
\end{equation*}
which is nothing new than was discussed previously. However, if we
assume further that the entries in $\bga^{(m)}$ are of the same order,
meaning that there exist positive numbers $\beta^-$ and $\beta^+$,
ordered as $\beta^- \le \beta^+$ and satisfying $\beta^+/\beta^- =
O(1)$, such that
\begin{equation}
  \beta^- \le |\mu(\bg_m,\bg_q)| \le \beta^+,
  \quad \forall ~ q \in \cS \setminus \{n(m)\},
  \label{eq:Ort3}
\end{equation}
then we have
\begin{align}
  \cI_{N_v} \le \beta^+ \sqrt{|\cS|-1} = \frac{\beta^+\big[\beta^{-}
      (|\cS|-1)\big]}{\beta^{-} \sqrt{|\cS|-1}} \le
  \frac{\beta^+\|\bga^{(m)}\|_1}{\beta^{-} \sqrt{|\cS|-1}} \le
  \frac{\beta^+ \cI_1}{\beta^{-} \sqrt{|\cS|-1}} =
  O\left(\frac{\cI_1}{\sqrt{|\cS|-1}}\right).
\label{eq:Ort4}
\end{align}

Recalling the discussion below definition \eqref{eq:th16} of the
semimetric $\cd$ and that $|\mu(\bg_m,\bg_q)| = 1 - \cd(m,q)$, we can
interpret  \eqref{eq:Ort3} as having points in 
$\Omega_\cS$  evenly distributed, at similar spacing.
If this condition holds, then $\cI_{N_v}$ is smaller than $\cI_1$, by  order $\sqrt{|\cS|}$. 
In practice, it may be difficult to have a large number $|\cS|$ of points at similar 
distance in the imaging plane, in order to see the improvement predicted by  \eqref{eq:Ort4}. 
However, this is just a bound, and the numerical simulations in section \ref{sect:gain} show that a significant 
reduction of $\cI_{N_v}/\cI_{1}$ is achieved even when the imaging region
is reduced to a  line. 

\subsubsection{Clusters of unknowns}
\label{sect:cluster}
The  multiple view interaction coefficient $\cI_{N_v}$ may be large 
for arbitrary distributions of points in $\Omega_\cS$, so we cannot
conclude from the estimates above that the
reconstruction $\bX^\ep$ approximates $\bX$.  However, if the 
points are  clustered around a few locations, indexed by the elements in the
set $\cC \subset \{1, \ldots, N_{\by}\}$ of cardinality $|\cC| \ll
|\cS|$, the  reconstruction is still useful, 
as we now show. 

The result follows by recasting Theorem \ref{thm:1} for the  new 
linear system 
\begin{equation}
  \bG \bU + \bcW = \bD_{\bW}, 
\label{eq:Cl5}
\end{equation}
with cluster unknown matrix $\bU$ and redefined "noise" 
$
\bcW = \bW + \bG \bR$, with  $\bR = \bX - \bU.$
The matrix $\bU$ is defined by projection of $\bX$ on the set of matrices with row support in   $\cC$, 
such that its restriction to the rows indexed by $\cC$
satisfies
\begin{equation}
  \bU_{\cC \rightarrow} = \bG_{\cC}^\dagger \bG \bX.
  \label{eq:Cl3}
\end{equation}
Here $\bG_{\cC}^\dagger$ is the pseudoinverse of  $\bG_{\cC} = (\bg_j)_{j\in \cC}$, the
restriction of the sensing matrix to the columns indexed in $\cC$, assumed to have 
full column rank. 
A similar calculation to that in \eqref{eq:calcul} implies that the "residual" $\bR$ satisfies 
$
\bG_{\cC}^\star \bG \bR = 0,
$
meaning that the columns of $\bG \bR$ are orthogonal to the range of $\bG_{\cC}$. In other words, 
$\bR$ accounts for the row support $\cS$ of $\bX$ being different from $\cC$. The magnitude of this 
residual depends on how close the points are clustered together, as stated in the next lemma proved in section \ref{sect:proof4}.

\vspace{0.05in}
\begin{lemma}
\label{lem:Cl1}
Decompose the set $\cS$ in $|\cC|$ disjoint
parts, called ``cluster sets'', indexed by the entries in $\cC$,
\begin{equation}
  \cS = \bigcup_{j \in \cC} \mathscr{S}_j, \quad
  \mathscr{S}_j = \{ q \in \cS \mbox{ such that } \cd(q,j) < \cd(q,j'), ~
  \forall j' \in \cC, j' \ne j\}, \quad j \in \cC.
  \label{eq:Cl1}
\end{equation}
Suppose that each cluster 
set $\mathscr{S}_j$  is supported within a ball of radius $\rc$ around
the point $j \in \cC$, with respect to the semimetric $\cd$, for
all $j \in \cC$, and that $\cd(j,j') > \rc$ for all distinct  $j', j \in \cC$.
Then, 
\begin{equation}
\|\bG \bR\|_F \le \sqrt{2 \rc} \|\bX^T\|_{2,1},
\label{eq:Cln1}
\end{equation}
where the index $T$ denotes the transpose.
\end{lemma}

The next theorem, proved in section \ref{sect:proof4}, is the
extension of Theorem \ref{thm:1}. It says that 
if the cluster radius $\rc$ and the cluster multiple view interaction coefficient 
\begin{equation}
  \cI^{\bU}_{N_v} = \max_{1 \le j \le N_{\by}} \sup_{\bvr \in \CC^{1
      \times N_v}} \sum_{q \in \cC \setminus\{n(j)\}}
  |\mu(\bg_j,\bg_q)| |\mu(\bvr,\bu_{q\rightarrow})|
\label{eq:Cln3}
\end{equation}
are small, the MMV reconstruction is row supported near $\cC$. This is an 
improvement over the estimate in Theorem \ref{thm:1}, because $\cI^{\bU}_{N_v}$ is 
much smaller than $\cI_{N_v}$ when the points in $\cC$ are well separated. 

\vspace{0.05in}
\begin{theorem}
\label{thm:3}
Let $\bX^\ep$ be the minimizer of \eqref{eq:th11}, with $\ep$ chosen
large enough to satisfy 
\begin{equation}
\|\bcW\|_F = \|\bW + \bG \bR\|_F < \ep.
\label{eq:Cln2}
\end{equation}
 Decompose $\bX^\ep$ it in two parts
\begin{equation}
\bX^\ep = \bU^{\ep,r} + \bE^{\ep,r},
\label{eq:Cln4}
\end{equation}
where $\bU^{\ep,r}$ is row supported in the set $\cB_r(\cC)$, the $r$
vicinity of $\cC$ with respect to the semimetric $\cd$, and
$\bE^{\ep,r}$ is the error supported in the complement $\{1, \ldots,
N_\by\} \setminus \cB_r(\cC)$. This satisfies the estimate
\begin{equation}
    \|\bE^{\ep,r}\|_{1,2} \le \frac{2 \cI_{N_v}^{\bU}}{r} \|\bX^\ep\|_{1,2} +
    \frac{1}{r} \big\| \big(\bG^\star \bW^\ep\big)_{\cC \rightarrow}
    \big\|_{1,2},
    \label{eq:Cln5}
\end{equation}
with $\bW^\ep = \bG(\bX^\ep - \bX)$ defined as in Theorem \ref{thm:1}. 
\end{theorem}

An extension of the quantitative estimate in Theorem \ref{thm:2} is possible, 
but we omit it here for brevity. The result says that we should expect 
a good qualitative agreement between $\bX^\ep$ and a local aggregate 
of $\bX$ over the cluster sets, if the single view interaction 
coefficient $\cI^{\bU}_{1}$ is small, meaning that the points in $\cC$ are 
sufficiently far apart.

\section{SAR imaging of direction dependent reflectivity}
\label{sect:applic1}
In this section we consider the application of SAR imaging of
direction dependent reflectivities. We begin with the data model in
section \ref{sect:applic1.1}, and then derive in section
\ref{sect:applic1.2} the linear system \eqref{eq:th1}. The discussion
in these two sections is similar to that in
\cite{borcea2016synthetic}, so we keep it short and give only the
information that is needed to connect to the theory in section
\ref{sect:theory2}.  We explore in section \ref{sect:applic1.3} the
condition of orthogonality of the rows of  $\bX$,
assumed in Proposition \ref{prop:1}, and use numerical simulations in
section \ref{sect:applic1.4} to illustrate the theoretical results.

\subsection{The SAR data model}
\label{sect:applic1.1}
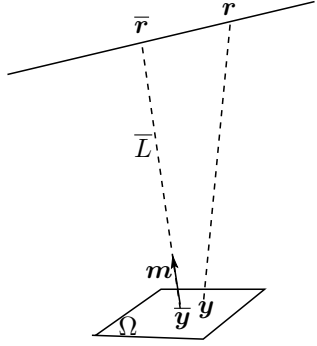
\begin{figure}[t]
\begin{center}
\includegraphics[width=0.25\textwidth]{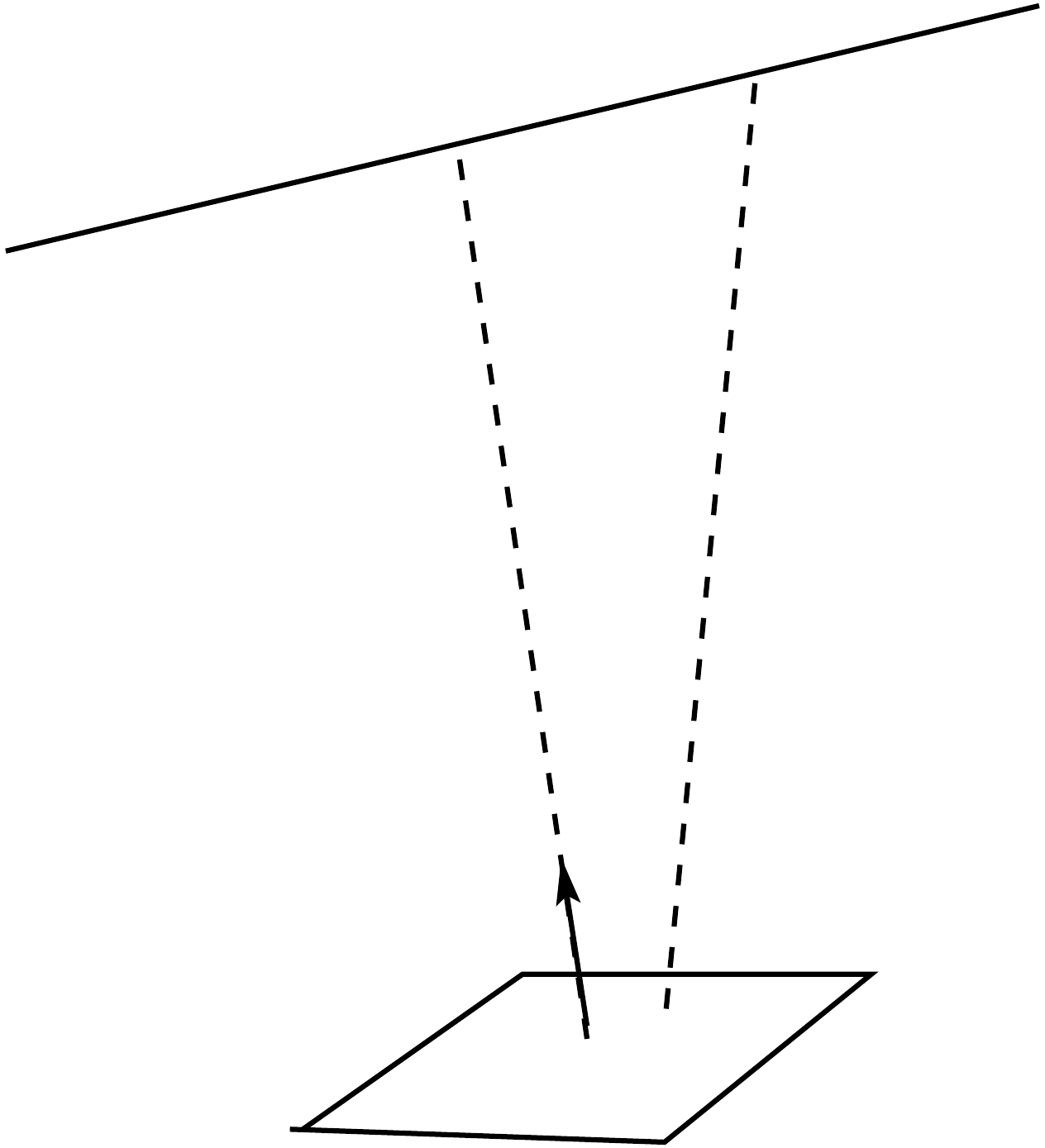}
\end{center}
\setlength{\unitlength}{3947sp}%
\begingroup\makeatletter\ifx\SetFigFont\undefined%
\gdef\SetFigFont#1#2#3#4#5{%
  \reset@font\fontsize{#1}{#2pt}%
  \fontfamily{#3}\fontseries{#4}\fontshape{#5}%
  \selectfont}%
\fi\endgroup%
\begin{picture}(9770,1095)(1568,-1569)
\put(5300,700){\makebox(0,0)[lb]{\smash{{\SetFigFont{7}{8.4}{\familydefault}{\mddefault}{\updefault}{\color[rgb]{0,0,0}{\normalsize $\oL$}}%
}}}}
\put(5550,-350){\makebox(0,0)[lb]{\smash{{\SetFigFont{7}{8.4}{\familydefault}{\mddefault}{\updefault}{\color[rgb]{0,0,0}{\normalsize $\oby$}}%
}}}}
\put(5370,-60){\makebox(0,0)[lb]{\smash{{\SetFigFont{7}{8.4}{\familydefault}{\mddefault}{\updefault}{\color[rgb]{0,0,0}{\normalsize ${\bmm}$}}%
}}}}
\put(5300,1480){\makebox(0,0)[lb]{\smash{{\SetFigFont{7}{8.4}{\familydefault}{\mddefault}{\updefault}{\color[rgb]{0,0,0}{\normalsize $\obr$}}%
}}}}
\put(5850,1590){\makebox(0,0)[lb]{\smash{{\SetFigFont{7}{8.4}{\familydefault}{\mddefault}{\updefault}{\color[rgb]{0,0,0}{\normalsize ${\br}$}}%
}}}}
\put(5700,-280){\makebox(0,0)[lb]{\smash{{\SetFigFont{7}{8.4}{\familydefault}{\mddefault}{\updefault}{\color[rgb]{0,0,0}{\normalsize ${\by}$}}%
}}}}
\put(5200,-420){\makebox(0,0)[lb]{\smash{{\SetFigFont{7}{8.4}{\familydefault}{\mddefault}{\updefault}{\color[rgb]{0,0,0}{\normalsize $\Omega$}}%
}}}}
\end{picture}%
\vspace{-1in}
\caption{Setup for SAR imaging using a linear synthetic aperture
  centered at $\obr$, at distance $\oL$ from the center
  $\oby$ of the imaging region $\Omega$. The antenna locations
  $\br$ span the aperture of length $a$, and $\by$ denotes a point in
  $\Omega$. The unit vector $\bmm = (\obr-\oby)/\oL$ pointing from $\oby$ to
  $\obr$ defines the range direction.}
\label{fig:setup}
\end{figure}

Consider the set-up illustrated in Figure \ref{fig:setup}, where we
display a piece of the synthetic aperture spanned by the moving
transmit-receive antenna, called a sub-aperture. We approximate the
sub-aperture by a line segment along the unit vector $\btau$, with
center at location $\obr$, and length $a$. The imaging region $\Omega$
lies on a plane surface, and is centered at location $\oby$, at
distance $\oL = |\obr-\oby|$ from the aperture center $\obr$.  The
antenna emits periodically the signal
$f(t)$ and measures the back-scattered waves. The waves propagate
much faster than the antenna, so we assume that the emission and
reception occur at the same location. The antenna moves by a small
increment $ \Delta \br = \frac{a}{(N_{\br}-1)} \btau $ between two
emissions, so the measurements are at locations
$
\br_j = \obr - \frac{a \btau}{2} + (j-1)\Delta \br$, for $ j = 1, \ldots,
N_{\br}.$

In the single scattering (Born) approximation, and neglecting for now
polarization effects, the scattered wave  at $\br_j$ is given by
\begin{align}
p(\br_j,t;\obr,\oom) = \int \frac{d \om}{2 \pi} \, e^{-i \om t} \hat f(\om)
k^2(\om) \sum_{q=1}^{N_\by} \rho_q(\obr,\oom)
\frac{\exp \big[ 2 i k(\om) |\br_j-\by_q| \big]}{\big(4 \pi
  |\br_j -\by_q|\big)^2}.
\label{eq:SA1}
\end{align}
Here the hat denotes Fourier transform with respect to time, $\om$ is
the frequency and  $\rho_q(\obr,\oom)$ is the reflectivity\footnote{The
  reflectivity is assumed slowly changing so it can be
  approximated by a constant over this sub-aperture and bandwidth. }
of the scatterer at $\by_q \in \Omega$. This depends on  the sub-aperture
center  $\obr$ and the central frequency $\oom$ of the signal $f$. 
The integral over $\om$ is over the support $|\om - \oom| \lesssim b$ of $\hat f$, where $b$ is the bandwidth.
The propagation of the waves between the antenna
location $\br_j$ and $\by_q$ is modeled with the Green's function for
Helmoltz's equation in the medium with constant wave speed $c$, and  
the wavenumber is $k(\om) = \om/c$.

In SAR imaging, the wave-field \eqref{eq:SA1} is convolved with the
time reversed emitted pulse, delayed by the round trip travel time of
the waves between the antenna and the center point $\oby$ in
$\Omega$. This data processing is called down-ramping
\cite{curlander1991synthetic} and we denote the result by
\begin{align}
d(\br_j,t;\obr,\oom) &= p(\br_j,t;\obr,\oom) \star_t {f^\star
\big(-t - 2|\br_j-\oby|/c \big)},
\label{eq:SA2}
\end{align}
where $f^\star$ denotes the complex conjugate of $f$.  The convolution $f(t) \star_t f^\star(-t)$
is called the pulse compressed signal. We denote it by $\varphi(bt) = f(t) \star_t f^\star(-t)$, 
with function $\varphi$ of dimensionless argument. This is supported at $t = O(1/b)$.

Let us define the unit vector 
$
\bmm = (\obr-\oby)/{\oL}
$
which determines the so-called range direction in imaging, and  
the orthogonal
projection $\mathbb{P} = {\bf I} - \bmm \bmm^T$  in the cross-range plane, orthogonal to $\bmm$. 
The size of the imaging region in the range and cross-range direction is given by the length scales 
\[
Y = \sup_{\by \in \Omega} |(\by-\oby) \cdot \bmm|, \qquad
Y^\perp = \sup_{\by \in \Omega} |\mathbb{P}(\by-\oby)|.
\]
We assume a typical  imaging regime defined by the scale order $\oL \gg a > Y^\perp \gg \ola$ and Fresnel numbers
\begin{equation}
\frac{a^2}{\ola\, \oL} 
\gtrsim \frac{(Y^\perp)^2}{\ola \,\oL} \gtrsim 1,
\label{eq:SA13}
\end{equation}
where $\ola = 2 \pi /\ok$ is the central wavelength and $\ok = k(\oom) = \oom/c$.  
These  inequalities mean 
physically that the wave front observed at the
sub-aperture or in $\Omega$ is not planar. If this were not the case,
it would be impossible to localize the scatterers in cross-range. 

Since the cross-range resolution  of classic SAR imaging \cite{cheney2009fundamentals} equals $\ola\, \oL/a$,  the 
inequalities \eqref{eq:SA13} ensure that $Y^\perp$ is larger than this limit, so image focusing can be observed. 
The range resolution  is determined by the accuracy of
travel time estimation from the down-ramped data \eqref{eq:SA2}. It is
of the order $c/b$, so typically  $Y \gtrsim {c}/{b}.$ 
In most imaging systems $b \ll \oom $. To simplify the data model, we assume a bandwidth and aperture segmentation in 
small enough sub-bands $b$  and sub-aperture sizes $a$ so that
\begin{equation}
\frac{b}{\oom} \frac{a Y^\perp}{\ola \,  \oL} \ll 1, \quad \frac{a^2 Y}{\ola \, \oL^2} \ll 1, \quad 
\frac{a^2 Y^\perp}{\ola \, \oL^2} \ll 1.
\label{eq:SA15}
\end{equation} 
Under these scaling assumptions and using the approximations described in \cite[Section 3.1]{borcea2016synthetic},
we can write  \eqref{eq:SA2} in the form 
\begin{equation}
\mathscr{D}_j(\obr) = \sum_{q=1}^{N_\by} \frac{\exp \Big[-2 i \ok 
  \frac{\Delta \br_j \cdot \mathbb{P} \Delta \by_q}{\oL} \Big]}{\sqrt{N_r}}
\mathscr{X}_q(\obr),
\label{eq:SA19}
\end{equation}
with the notation 
$
\Delta \br_j = \br_j -\obr$ and  $\Delta \by_q = \by_q -\oby. $
Here $\mathscr{D}_j(\obr) $ are the down-ramped data \eqref{eq:SA2}, up to some scaling factor, and  evaluated at 
a fixed time $\bar t$,
\begin{equation}
\mathscr{D}_j(\obr) = d(\br_j,\ot;\obr,\oom) e^{i \ok \,\ot} \left(\frac{4
  \pi \oL}{\ok}\right)^2,
\label{eq:SA21}
\end{equation}
whereas
\begin{equation}
\mathscr{X}_q(\obr) = \rho_q(\obr,\om) \sqrt{N_r} \varphi \Big[ b
  \Big(\ot + \frac{2 \bmm \cdot \Delta \by_q}{c} \Big)\Big] \exp \Big[ - 2
    i \ok \Big(\bmm \cdot \Delta \by_q - \frac{\Delta \by_q \cdot
      \mathbb{P} \Delta \by_q}{2\oL}\Big)\Big].
\label{eq:SA20}
\end{equation}
We suppressed all the constant variables in the arguments of 
$\mathscr{D}_j$. By fixing the time $\bar t$, we limit the sum in \eqref{eq:SA19} to 
the set of points with range coordinates $ \bmm \cdot
\Delta \by_q = -t + O(c/b).  $ This set is called a range bin in the
SAR literature \cite{curlander1991synthetic}.
We consider a single range bin,  and study
the estimation in the cross-range direction of the reflectivity, for
the single frequency sub-band centered at $\oom$.  
\subsection{The MMV formulation}
\label{sect:applic1.2}
The multiple views  correspond to different
sub-apertures of size $a$, dividing a larger aperture of size $A$. The sub-apertures 
are centered at $\obr_v$, for $v = 1, \ldots, N_v$.
The noiseless data model for the $v$--th view is 
\eqref{eq:SA19}, with $\obr$ replaced by $\obr_v$,$\oL$ replaced
by $\oL_v = |\obr_v-\oby|$, $\bmm$ replaced by $\bmm_v  = |\obr_v-\oby|/\oL_v$ and 
$\mathbb{P}$ replaced by $\mathbb{P}_v = \itbf{I} - \bmm_v \bmm_v^T$. We 
assume for simplicity that the large aperture is linear, along the unit vector $\btau$.

Under technical scaling assumptions described in detail in \cite{borcea2016synthetic}, 
which mean physically that the  imaging points remain
within the same classic SAR resolution limits  for all the views, 
we obtain  from \eqref{eq:SA19}
the linear system \eqref{eq:th1}, for matrices $\bD$, $\bX$ and $\bG$
with entries
\begin{equation}
D_{j,v} = \mathscr{D}_j(\obr_v), \quad X_{q,v} = \mathscr{X}_{q}(\obr_j), 
\quad G_{j,q} = \frac{1}{\sqrt{N_r}} \exp \Big[-2 i \ok 
  \frac{\Delta \br_j \cdot \mathbb{P}_1 \Delta \by_q}{\oL_1} \Big].
\label{eq:SA23}
\end{equation}
Note that the sensing matrix $\bG$ is defined relative to the first
sub-aperture. Its columns $\bg_q$, for $q = 1, \ldots, N_\by$, have norm one, as assumed in
\eqref{eq:normg}, and their correlation
\begin{equation}
\mu(\bg_q,\bg_l) = \sum_{j = 1}^{N_r} G_{j,q}^\star G_{j,l} =
\frac{1}{N_r} \sum_{j=1}^{N_r} \exp \Big[-2 i \ok \frac{\Delta \br_j
    \cdot \mathbb{P}_1 (\by_q - \by_l)}{\oL_1} \Big]
\label{eq:SA24}
\end{equation}
is a function of $\by_q-\by_l$, as stated below equation
\eqref{eq:th16}.  We can approximate further this correlation by
replacing the sum with the integral over the sub-aperture,
\begin{equation}
\mu(\bg_q,\bg_l) \approx \frac{1}{a} \int_{-a/2}^{a/2} dr \, \exp
\Big[-2 i \ok r \frac{ \btau \cdot \mathbb{P}_1 (\by_q -
    \by_l)}{\oL_1} \Big] =\mbox{sinc} \Big[\frac{\ok a \btau \cdot
  \mathbb{P}_1 (\by_q - \by_l)}{ \oL_1} \Big].
\label{eq:SA25}
\end{equation}
This attains its maximum, equal to $1$, when $q = l$, and satisfies
$|\mu(\bg_q,\bg_l)< 1$ for all $q \ne l$, as assumed in
\eqref{eq:th16b}.  Moreover, $|\mu(\bg_q,\bg_l)|$ decays monotonically
in the vicinity of its peak, so we can relate the Euclidian distance
between the points to the semimetric
$\cd(q,l)$, as pointed out below equation \eqref{eq:th16}.

\subsection{Orthogonality of the rows}
\label{sect:applic1.3}
To use the results in section \ref{sect:orthogX}, we now study under which conditions the rows $\bx_{q\rightarrow}$ of
$\bX$ are approximately orthogonal.
For this purpose, we assume that $\brho_q(\obr_v,\om)$ 
changes  slowly with $\obr_v$, on a length scale larger than $a$.  This is 
consistent with the MMV formulation, which approximates the reflectivity by a
constant for each sub-aperture. We
also suppose that the sub-apertures  overlap, with two consecutive centers separated by a small distance 
with respect to $a$. 
This allows us to approximate the sums in the correlations of  the rows by integrals 
over the large aperture of linear size $A$, centered at $\br_o$. 

\vspace{0.05in}
\begin{proposition}
\label{prop:SARort} There exists a constant $C_{q,l}$ that depends on 
how fast the reflectivities at points $\by_q$ and $\by_l$ change with
direction, such that
\begin{equation}
\big|\mu(\bx_{q \rightarrow}, \bx_{l \rightarrow})\big| \le \min\{1, C_{q,l}/|Q|\},
\quad \mbox{for} ~ q\ne l, ~ ~ q,l = q, \ldots, N_\by, \quad Q = 4 \pi \frac{A \btau \cdot \mathbb{P}_o (\by_q - \by_l)}{\ola |\br_o -
  \oby|},
\label{eq:propOrt1}
\end{equation}
where $\bmm_o = (\br_o-\oby)/|\br_o-\oby|$ and $\mathbb{P}_o = {\itbf I} - \bmm_o \bmm_o^T$.
\end{proposition}

\vspace{0.05in} This proposition, proved in Appendix \ref{ap:SARort},
shows that the correlation of the rows of the unknown matrix $\bX$ is
small for points that are separated in cross-range by distances larger than
$\ola |\br_o-\oby|/A$. This length scale is the cross-range resolution
of SAR imaging over the large aperture $A$.  It is also the distance
at which isotropic scatterers must be separated in order to guarantee
unique recovery of their reflectivity with $\ell_1$  (SMV) optimization
over the large aperture, as follows from
\cite{fannjiang2013compressive,chai2013robust,chai2014imaging,borceaKocyigit}.

In the linear system \eqref{eq:th1} with matrices \eqref{eq:SA23}, we
use multiple views from sub-apertures of size $a \ll A$. Each single view
corresponds to an SMV problem, and the condition of unique recovery
for that problem is known to be that the scatterers should be much further apart, at distance of
order $\ola |r_o-\oby|/a$. In MMV we use the entire large aperture,
segmented in $N_v$ smaller sub-apertures.

When the scatterers are approximately isotropic, the constant in
\eqref{eq:propOrt1} is $C_{q,l} \approx 2$. In this case there is no
need to segment the aperture, so it is natural to ask if the MMV
reconstruction is similar to the SMV one, over the large aperture.
This is a difficult question, but we can say from the results in
section \ref{sect:orthogX} that MMV will work better\footnote{As shown
  in section \ref{sect:orthogX}, the improvement is dependent on the
  distribution of the scatterers in the imaging region.} then SMV
over one sub-aperture, because the rows of the unknown matrix $\bX$
are approximately orthogonal when the points in its support are
at distances of order $\ola |r_o-\oby|/A \ll
\ola |r_o-\oby|/a.$ The numerical simulations in the next section demonstrate that
this is the case, as well.

When the scatterers have a stronger dependence on direction, the SMV
approach over the large aperture does not work well. Aperture
segmentation is needed to avoid systematic modeling errors in the
optimization. While we may apply the SMV approach for a single
sub-aperture, Proposition \ref{prop:SARort} and the results in section
\ref{sect:orthogX} show that the MMV method performs better.

\subsection{Numerical results}
\label{sect:applic1.4}
We present here numerical results that illustrate the
theory presented in section \ref{sect:theory2}. We begin in section
\ref{sect:gain} with a computational assessment of the reduction of
the multiple view interaction coefficient with respect to the single
view one, in the case of orthogonal rows of the unknown matrix
$\bX$.  Then we present in section \ref{sect:num1} imaging results,
using the parameters of the X-band GOTCHA SAR data set \cite{GOTCHA}:
The receive-transmit platform moves on a linear aperture $A = 1.5$km
at altitude $8$km, and with center $\obr_o$ at $7$km west of
$\oby$. The platform emits and receives signals every meter. The
central frequency is $10$GHz and since we only present imaging in
cross-range, the bandwidth plays no role. The waves propagate at 
speed $c = 3\cdot10^8$m/s.

The data are generated numerically using the single scattering
approximation. The additive noise matrix $\bW$ has mean zero and
independent complex Gaussian entries with standard deviation $\sigma$ given 
as a percent of the largest entry in $\bD$.
The optimization problem \eqref{eq:th11} is solved using the software
package CVX \cite{cvx}.

\subsubsection{Numerical illustration of effects of orthogonality of rows of $\bX$} 
\label{sect:gain}
The discussion in section \ref{sect:orthogX} says that if the points
in $\Omega_\cS$ are distributed evenly in the imaging
window $\Omega$, and the rows of $\bX$ are orthogonal, then the
multiple view interaction coefficient $\cI_{N_v}$ is smaller than
$\cI_1$, by a factor of order $\sqrt{|\cS|}$. Here we focus
attention on imaging in the cross-range direction, so the imaging
region is reduced to a line segment.  We cannot have a large
number of points with similar mutual separation on a line. Nevertheless, we
show that the numerically computed ratio $\cI_{1}/\cI_{N_v}$ increases with
$|S|$, at a  slightly slower rate than $\sqrt{|\cS|}$.

We display in Figure \ref{fig:illIvec} the ratio
$\cI_{1}/\cI_{N_v}$ computed for imaging scenes with $|S|$ ranging
from $4$ to $50$, and cross-range separation of nearby neighbors
chosen randomly, uniformly distributed in the interval $\big[\ola\, 
  \oL_o/A, 3 \ola\,  \oL_o /A\big]$, where $\oL_o = |\obr_o - \oby|$. The large aperture
$A$ is divided in sub-apertures of size $a = A/20$.  The rows of $\bX$
have length $50$ and are orthogonal, to stay within the setting
of section \ref{sect:orthogX}. 

The left plot in Figure \ref{fig:illIvec} shows the ratio
$\cI_{1}/\big(\cI_{N_v}\sqrt{|\cS|}\big)$ computed for one realization
of the imaging scene. We note that the increase of
$\cI_{1}/\cI_{N_v}$ with $|\cS|$ is slightly slower than
$\sqrt{|\cS|}$.  The histograms in Figure \ref{fig:illIvec}, computed
for $2500$ realizations of the imaging scene, also show that the ratio
is slightly less than $\sqrt{|\cS|}$.

\begin{figure}[t]
\begin{center}
\includegraphics[width=0.23\columnwidth]{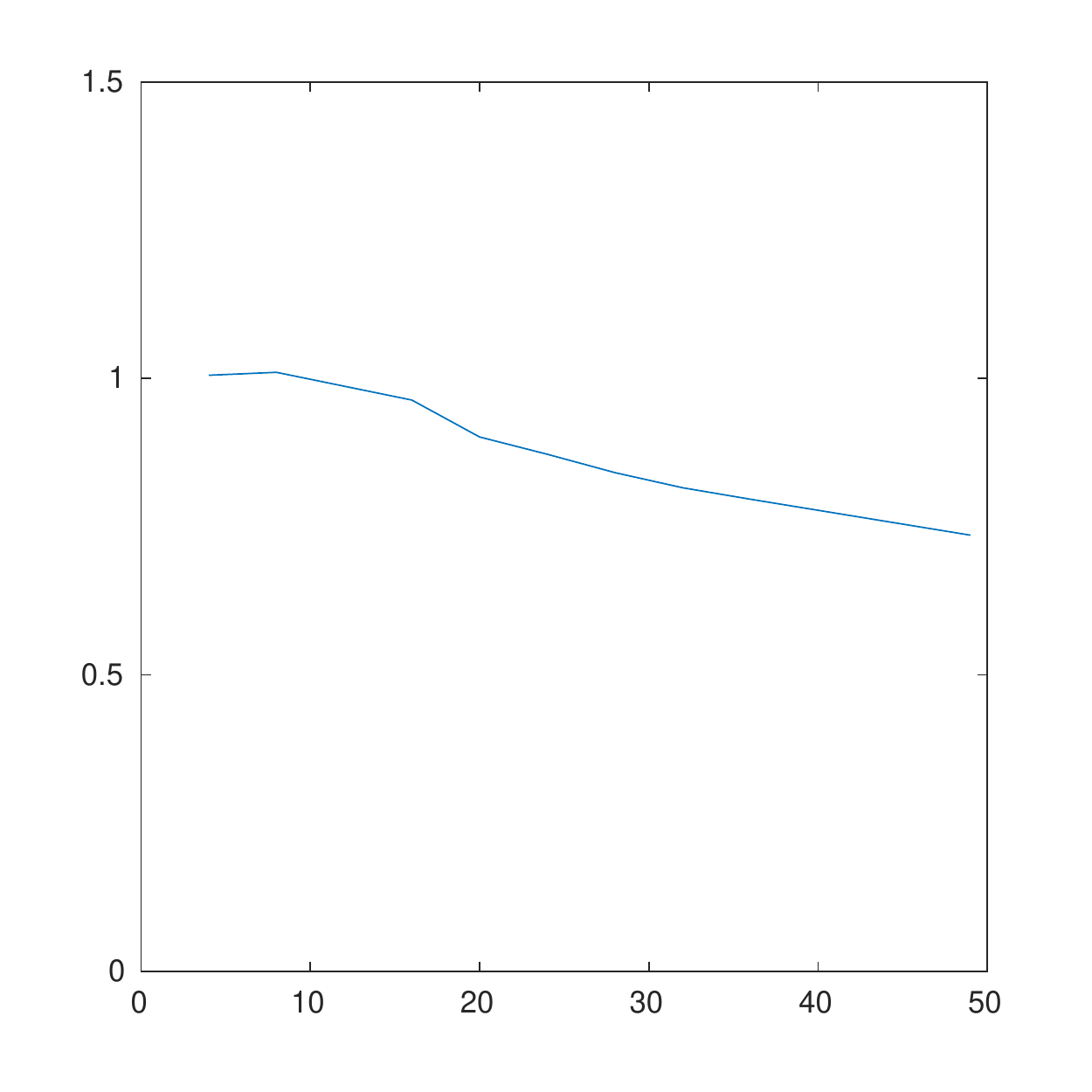}
\includegraphics[width=0.23\columnwidth]{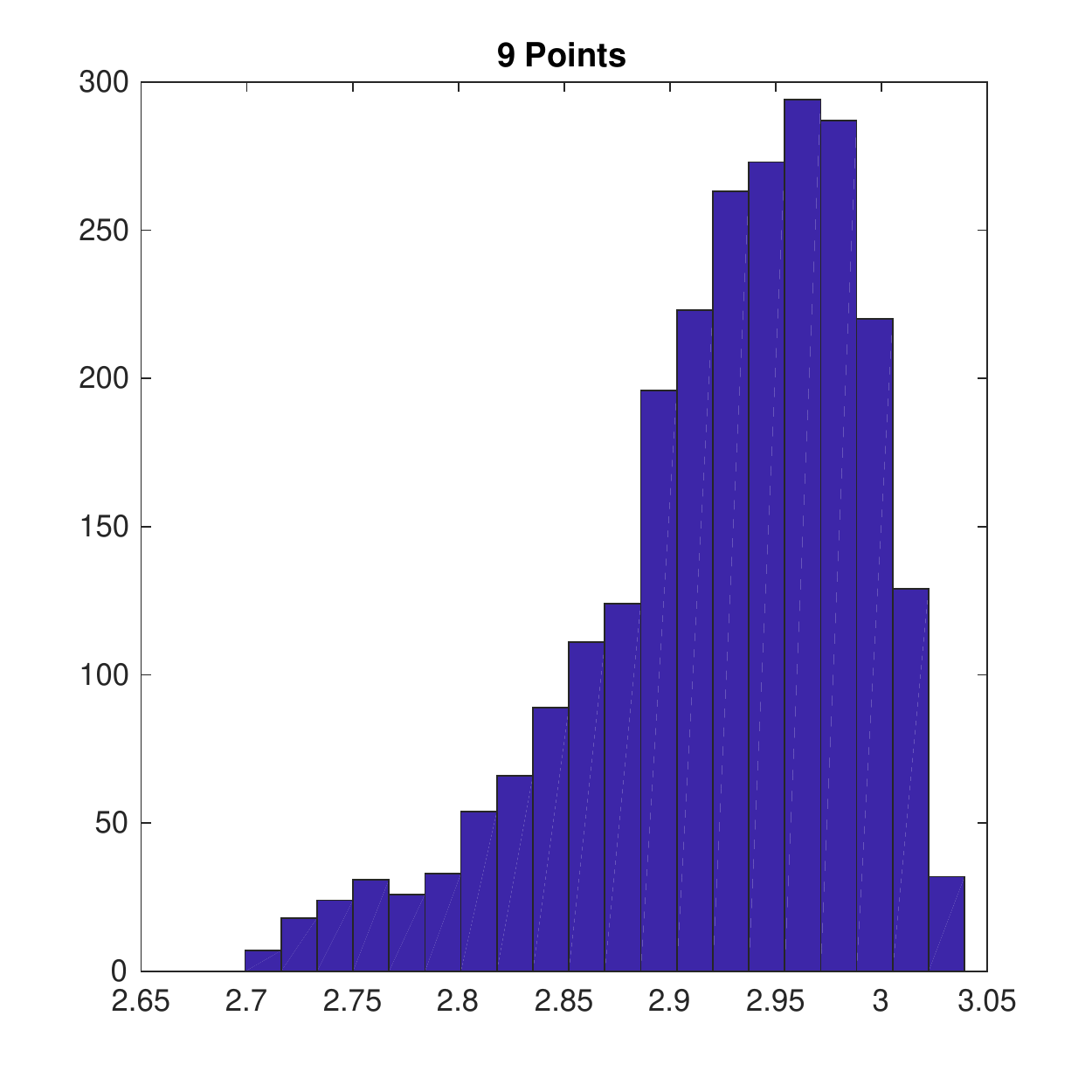}
\includegraphics[width=0.23\columnwidth]{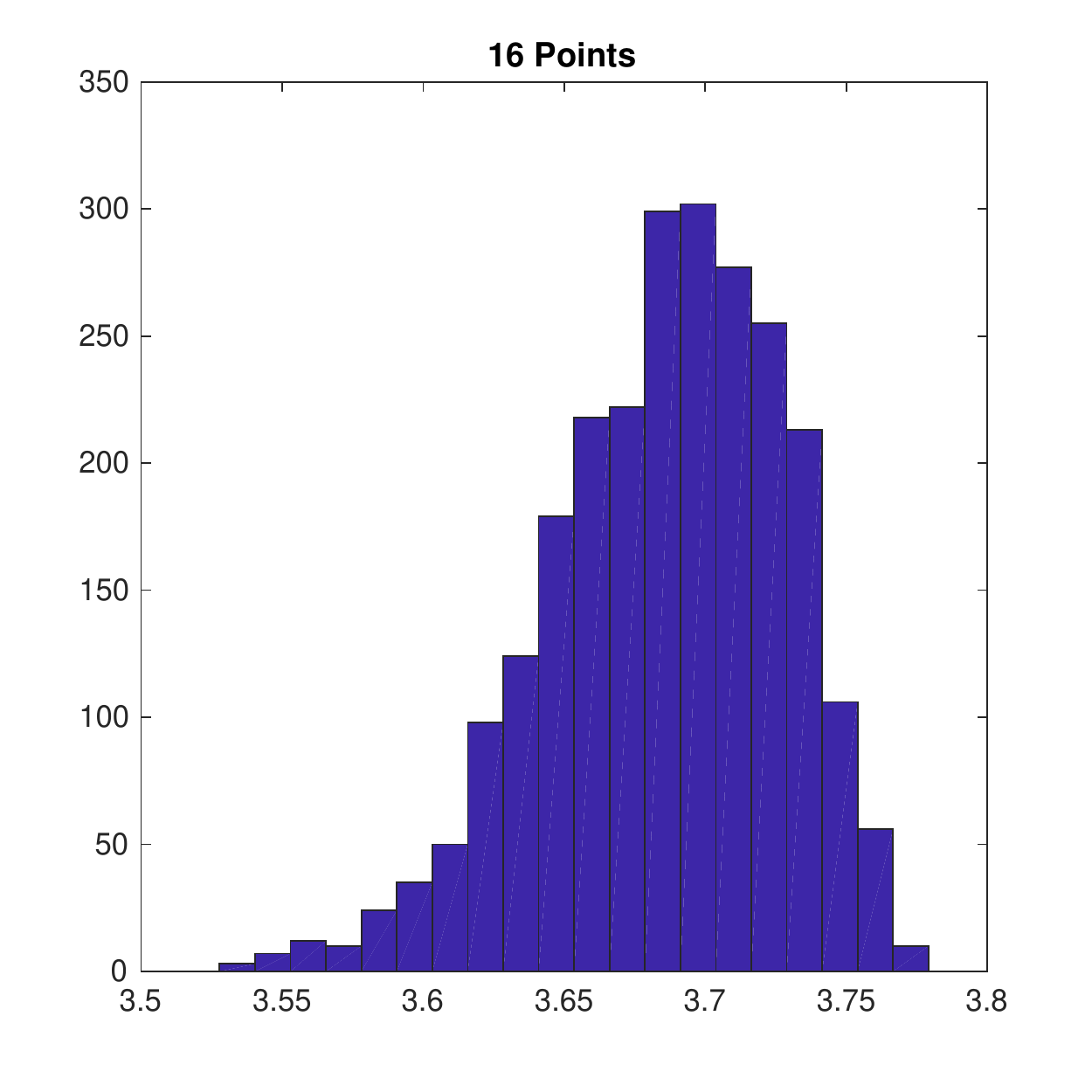}
\includegraphics[width=0.23\columnwidth]{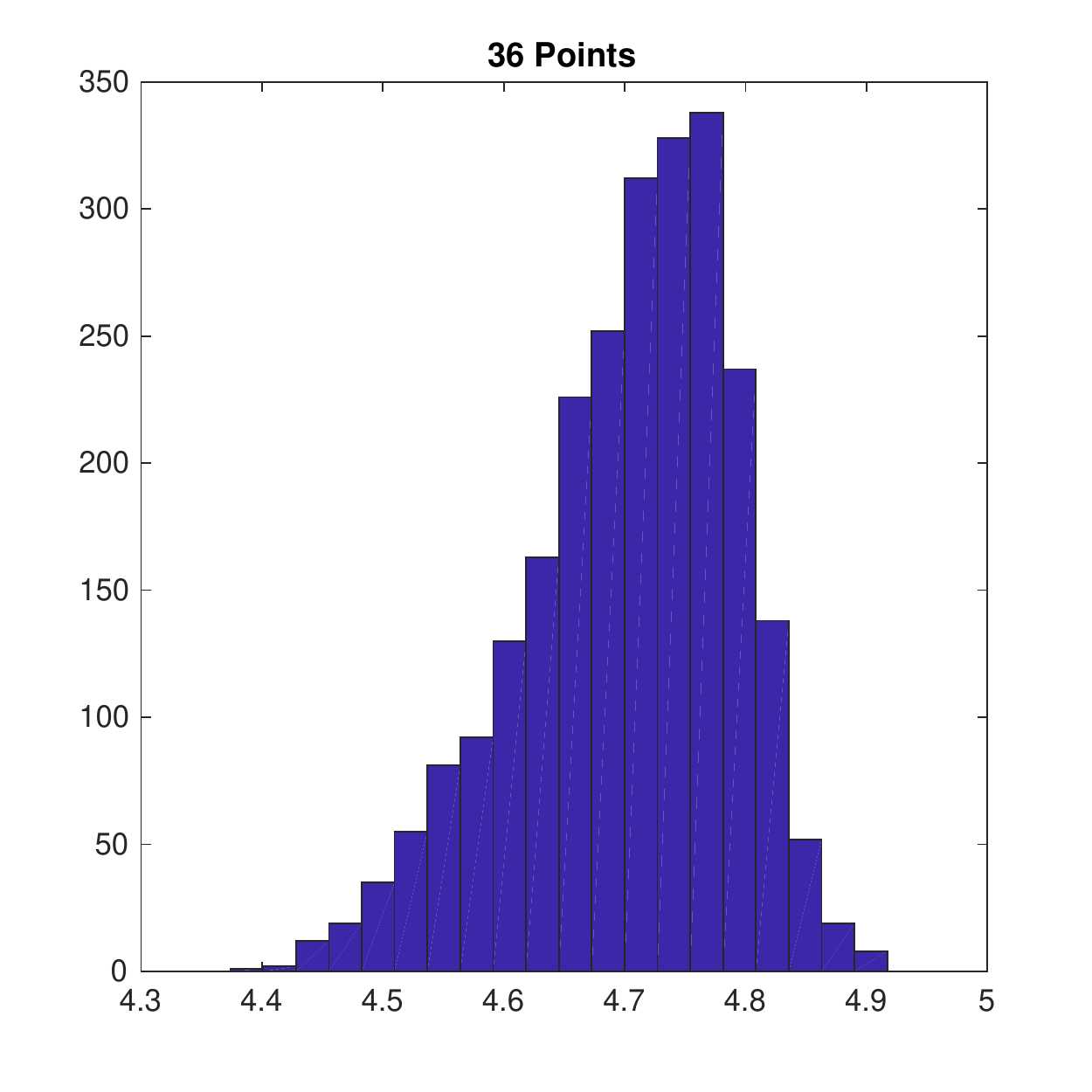}
\end{center}
\vspace{-0.12in}
\caption{Left plot: The ratio $\frac{\cI_{1}}{\cI_{N_v}\sqrt{|\cS|}}$ vs. $|S|$
  in the abscissa. The other plots: Histograms of the ratio ${\cI_{1}}/{\cI_{N_v}}$ for
  $2500$ realizations of the imaging scene. From left to right
  $|\cS|$ equals $9$, $16$ and $36$. The ordinate shows the number of
  realizations and the abscissa is the value of ${\cI_{1}}/{\cI_{N_v}}$.}
\label{fig:illIvec}
\end{figure}

\subsubsection{Imaging results}
\label{sect:num1}

\begin{figure}
\begin{center}
\includegraphics[scale=0.32]{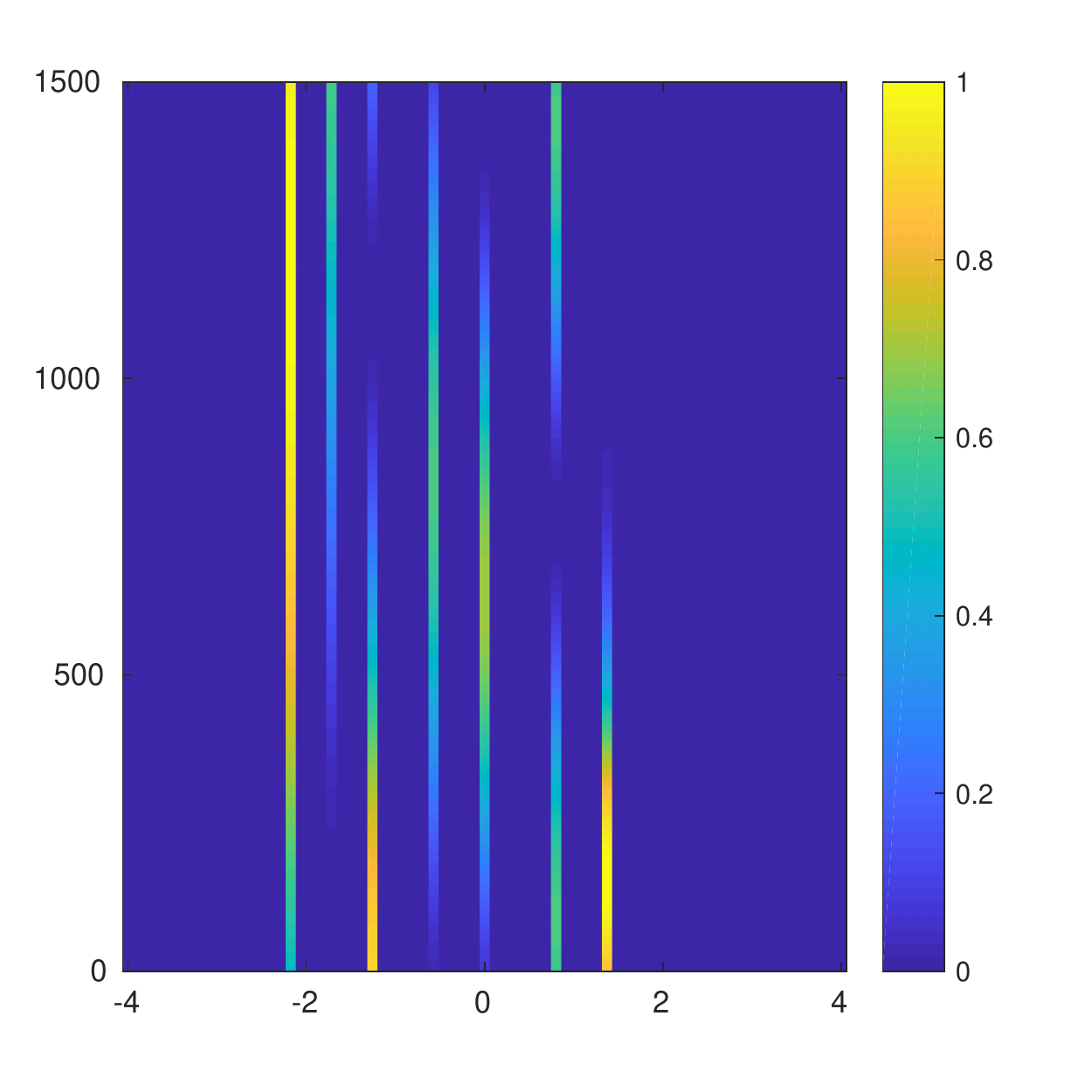}\hspace{-0.1in}
\includegraphics[scale=0.32]{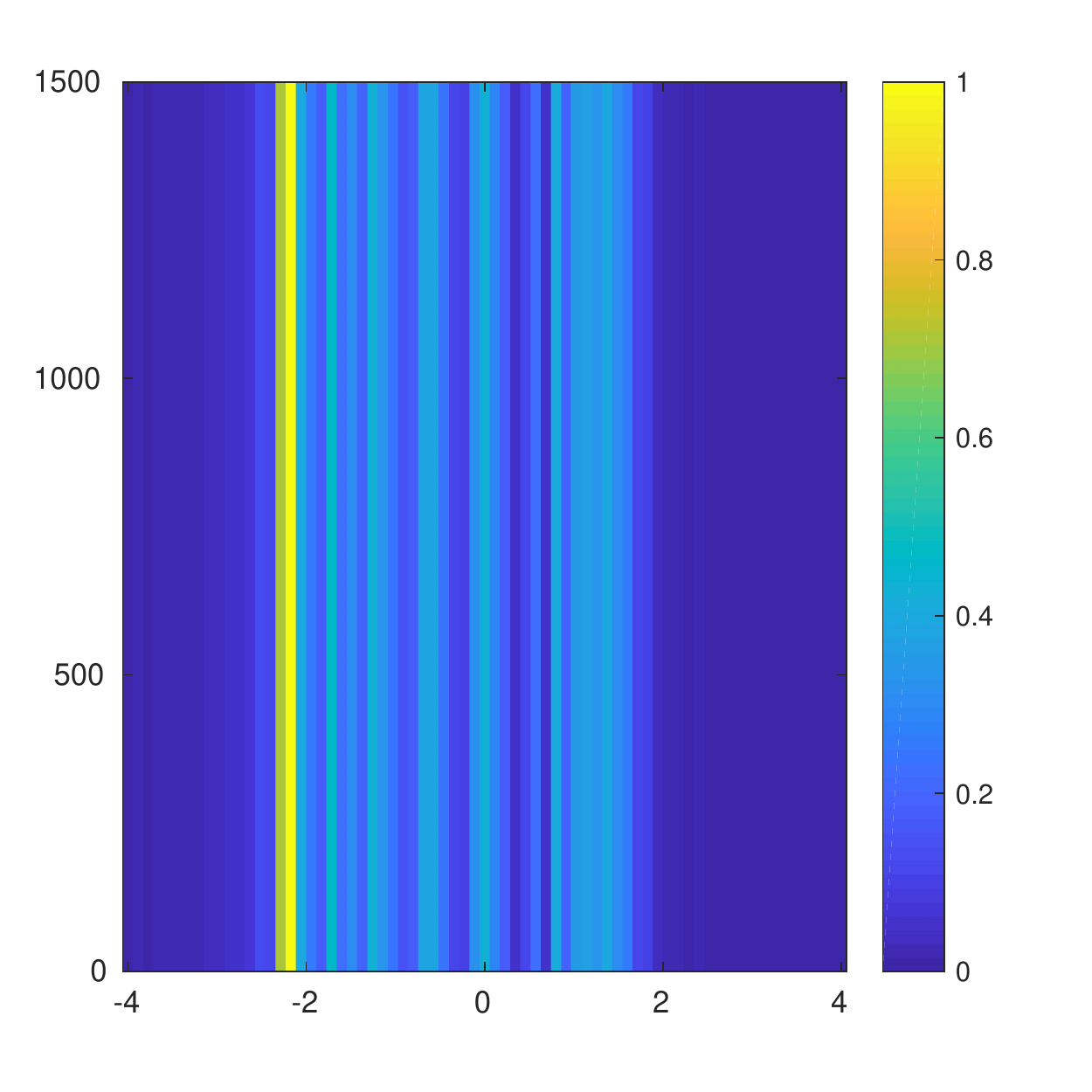}\hspace{-0.1in}
\includegraphics[scale=0.32]{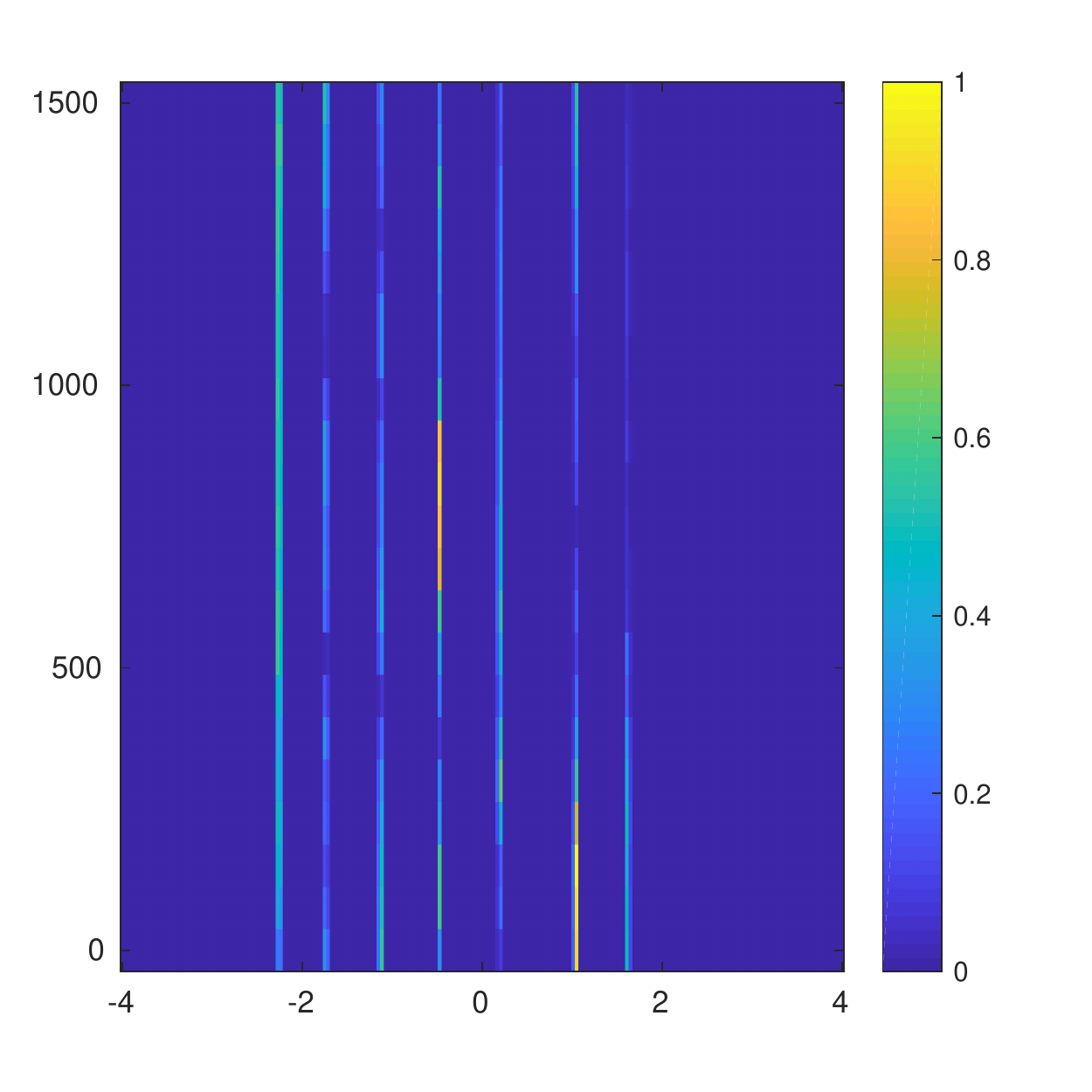}\hspace{-0.1in}
\includegraphics[scale=0.32]{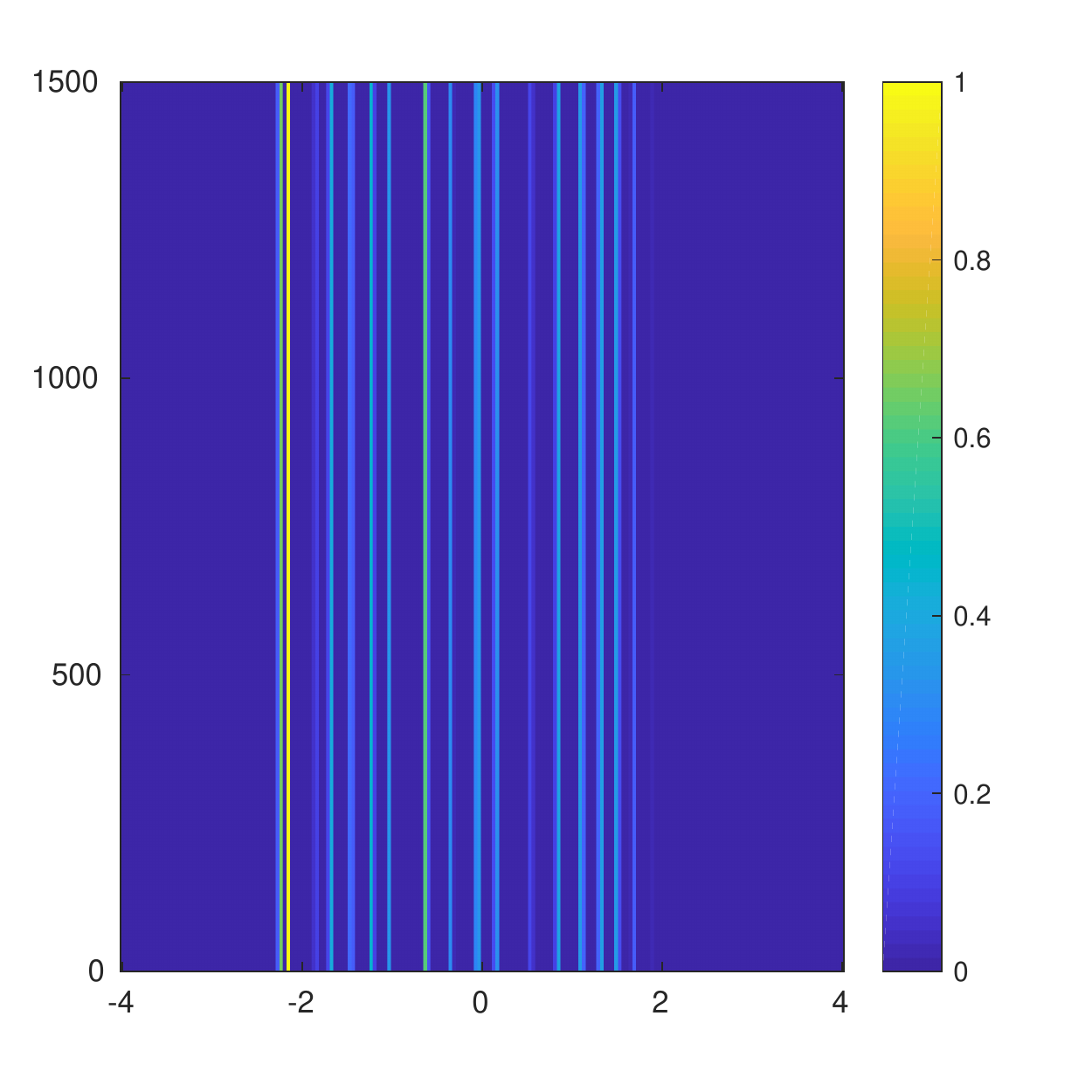}
\end{center}
\vspace{-0.12in}
\caption{From left to right:  (1) Exact reflectivity function as viewed from the
  location on the flight path (ordinate, in meters) vs.  the
  cross-range location in the imaging scene (abscissa, in units $\ola\,
  \oL_o/A$).  (2) The conventional SAR image
  \eqref{eq:SARmig} calculated over the entire aperture.  (3) The MMV reconstruction. (4) The SMV
  reconstruction.}
\label{fig:Image1}
\end{figure}

We begin with a comparison of imaging results obtained with the MMV
optimization formulation \eqref{eq:th11} for $N_v = 24$, the SMV
formulation for $N_v = 1$, and the conventional SAR image. The latter is given by the
superposition of the down-ramped data \eqref{eq:SA2}, synchronized using the 
round-trip travel time of the waves from the radar platform to the imaging point
\begin{equation}
I^{SAR}(\by;\obr) = \sum_{j = 1}^{N_r} d\big(\br_j,t=2|\br_j-\by|/c;\obr,\oom\big).
\label{eq:SARmig}
\end{equation}
The superposition may be over the entire aperture centered at $\obr =
\obr_o$, in which case $N_r = 1500$, or over a sub-aperture, centered at
$\obr = \obr_v$ for $v = 1, \ldots, N_v$, in which case $N_r = 300$.
The sub-aperture length is $a = A/6 = 300$m, and the spacing between
the sub-apertures is $50$m, center to center.  The results in Figures
\ref{fig:Image1}--\ref{fig:Image3} are for noiseless data and in
Figure \ref{fig:Image3Noise} we consider noise with 
standard deviation $\sigma = 10\%$.

The images in Figure \ref{fig:Image1} are obtained for a scene with $6$ small scatterers
at cross-range locations spaced by distances of approximately $\ola \, \oL_o/A$.
The exact reflectivity is shown in the left plot.
The SAR image \eqref{eq:SARmig} computed over the entire aperture $A =
1.5$km is shown in the second plot. Note that this treats the reflectivity as isotropic (i.e., constant 
along the ordinate). It does not resolve well the
location of the five scatterers that are visible only on about a sixth
of $A$, but it obtains a large peak for the one scatterer with
reflectivity that varies less with direction. 
The MMV image  recovers exactly the
support of the scatterers, whereas the SMV method has many spurious
peaks.  This is an illustration of the result in section \ref{sect:quant}, which says that MMV may give a
better estimate of the support of the scatterers. However, the
estimate of the value of the reflectivity is not accurate, unless the
scatterers are further apart.

In Figure \ref{fig:Image3} we consider reflectivities that vary more
rapidly over directions, and compare the effect of the size of the
sub-aperture on the quality of the reconstructions with the MMV
approach. The images show that the best
reconstruction is for $a = 70$m, which corresponds roughly with the
scale of variation of the true reflectivity in the top plot. For the
smaller aperture $a = 40$m (left, bottom plot) the reconstructed
support is close but not exact, whereas for the larger aperture $a =
100$m (right, bottom plot) the image has spurious peaks caused by the
systematic error due to the reflectivity varying on a smaller scale
than the sub-aperture. Thus, we conclude that in order to image
successfully direction dependent reflectivities, it is necessary to have
a good estimate of their scale of variation, so that the
aperture is properly segmented.

In Figure \ref{fig:Image3}  we display the effect of additive noise with standard deviation $\sigma = 10\%$ on  the MMV reconstruction of the reflectivity, 
for sub-aperture size $a = 70$m. 
We note that  for such  noise the support of the reconstruction is basically unchanged
and the values of the reflectivity are only slightly different. Naturally, at higher noise levels,
the reconstruction will be worse.

\begin{figure}[t]
\begin{center}
\includegraphics[scale=0.32]{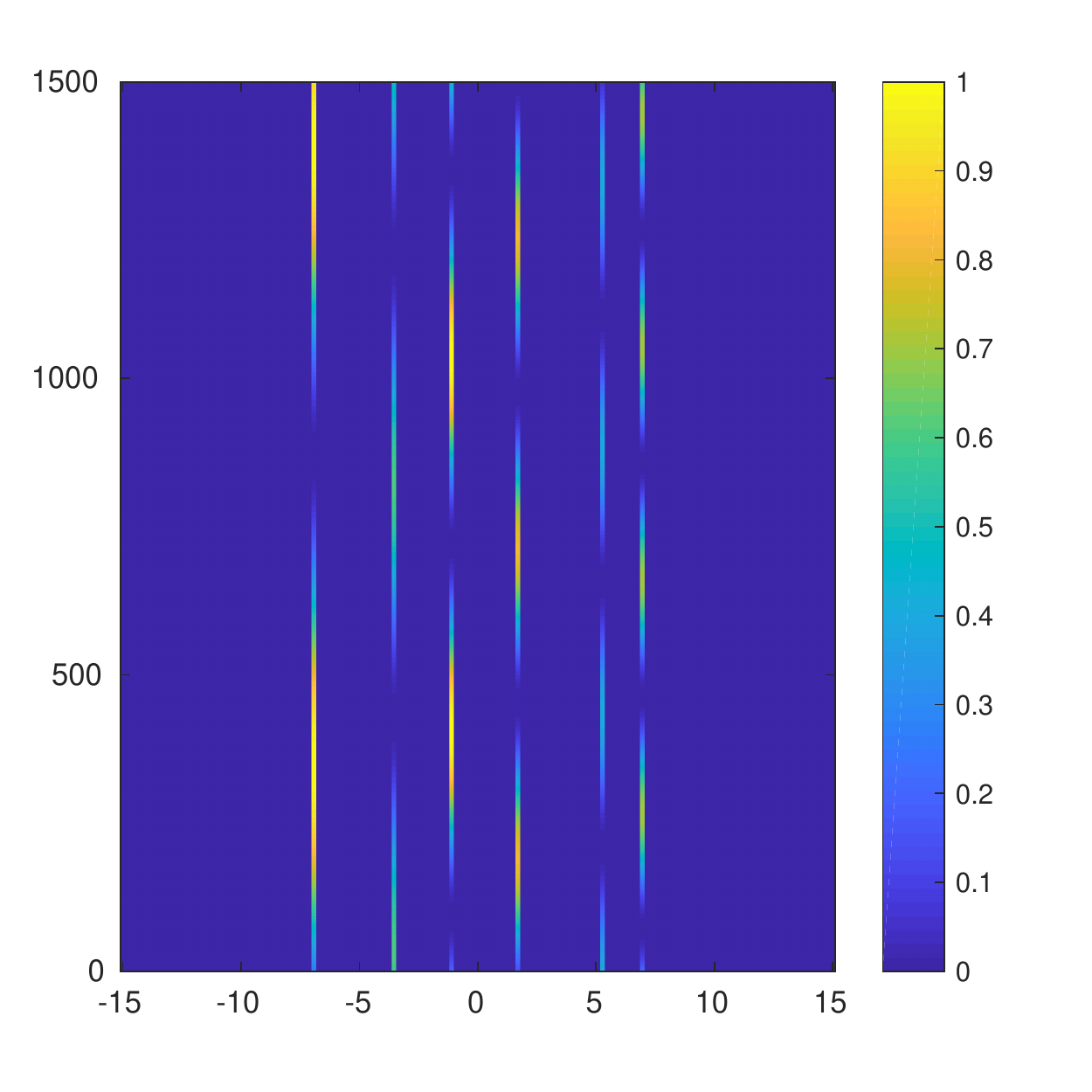} \hspace{-0.1in}
\includegraphics[scale=0.32]{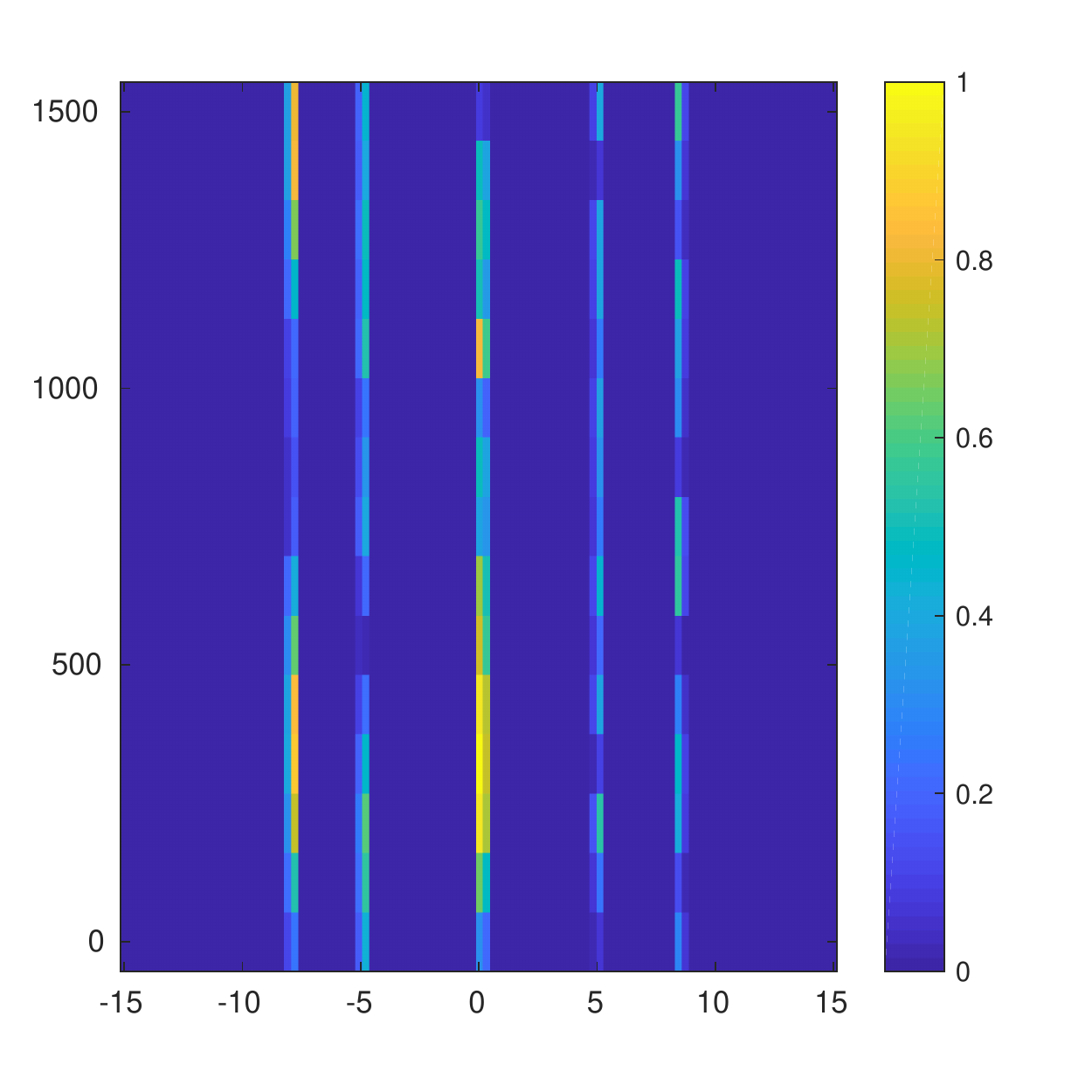}\hspace{-0.1in}
\includegraphics[scale=0.32]{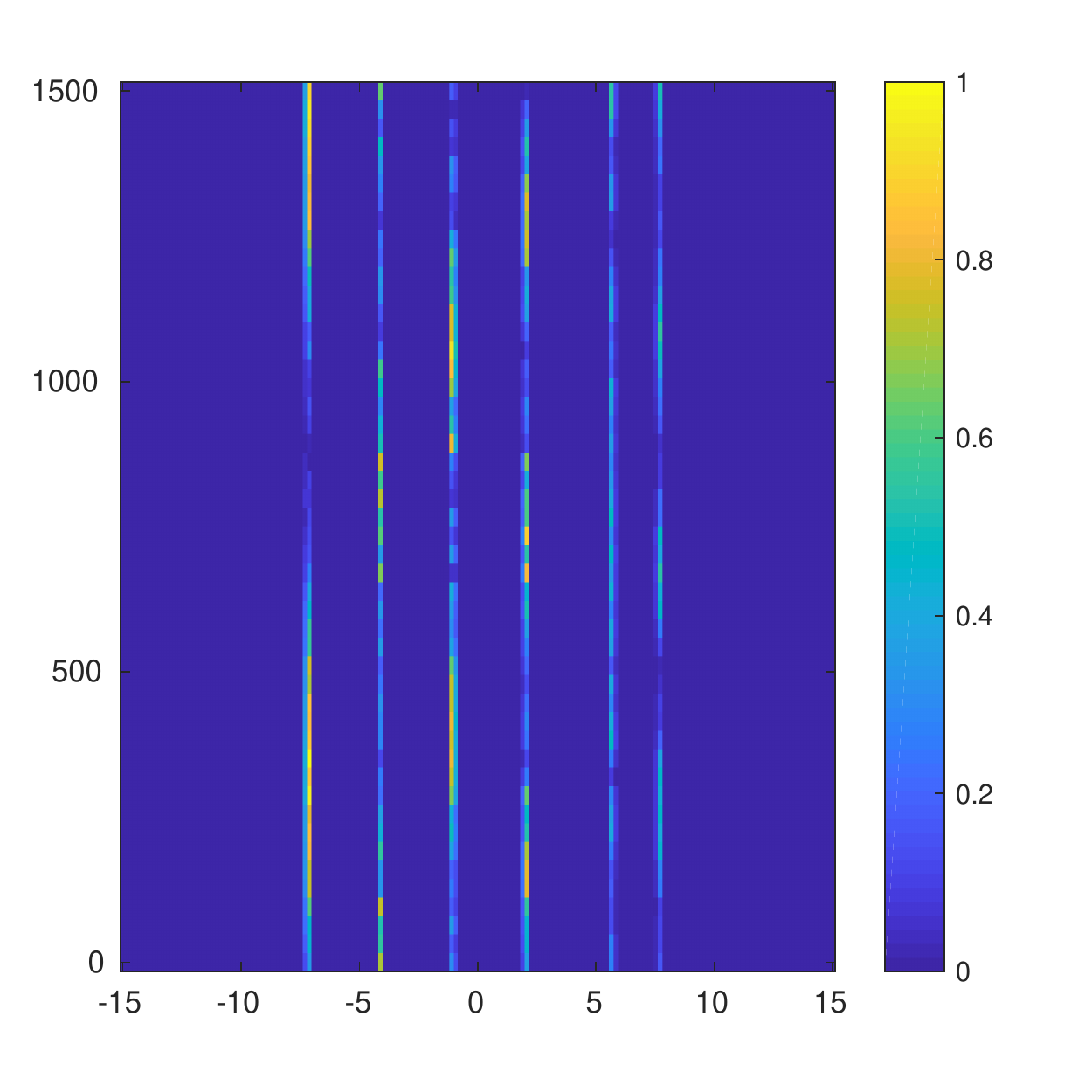}\hspace{-0.1in}
\includegraphics[scale=0.32]{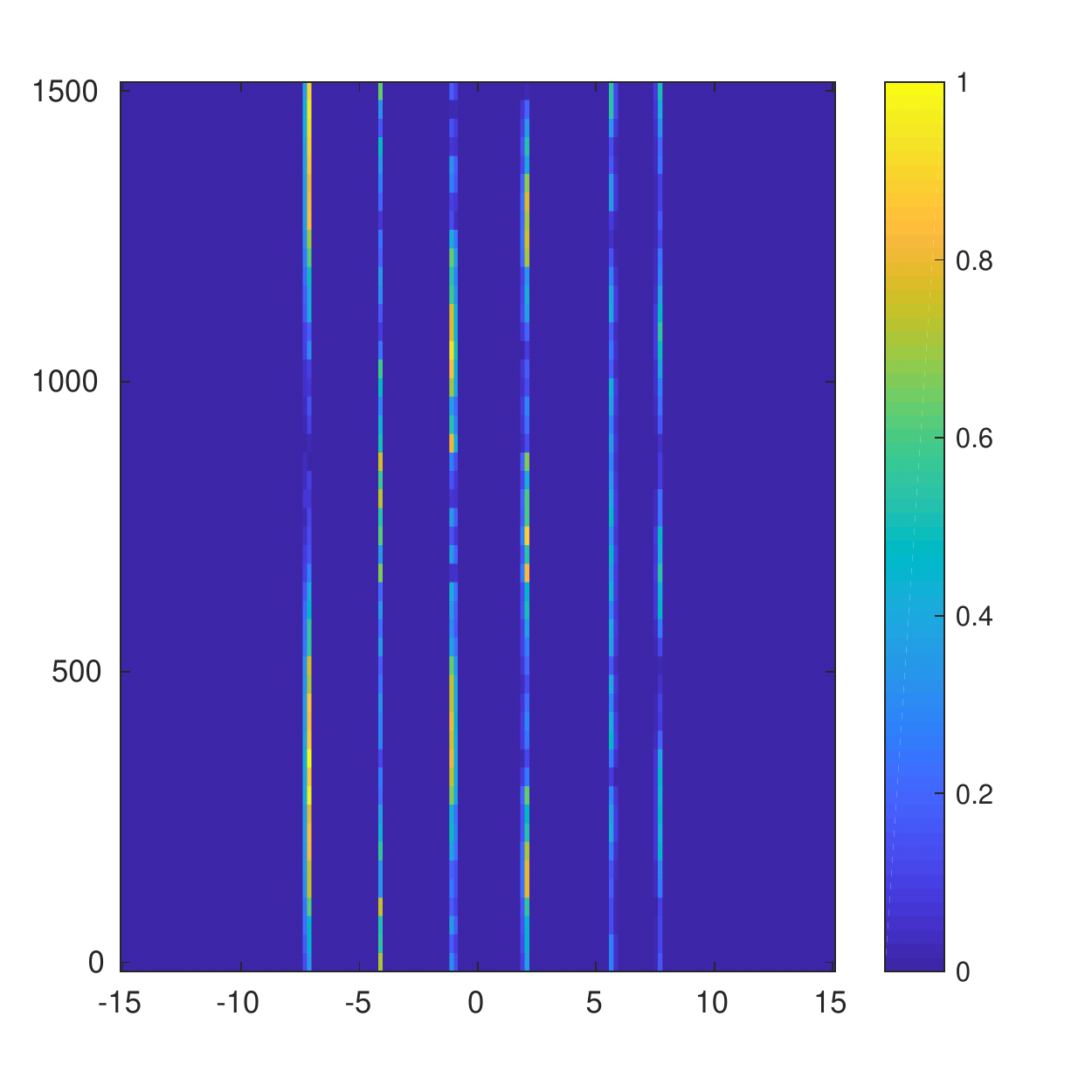}
\end{center}
\vspace{-0.12in}
\caption{Left plot: Exact reflectivity function as viewed from the
  location on the flight path (ordinate, in meters), vs.  the
  cross-range location in the imaging scene (abscissa, in units $\ola \, 
  \oL_o/A$).  Other plots: The MMV reconstruction for apertures $a =
  50$m, $70$m and $100$m, from left to right. }
\label{fig:Image3}
\end{figure}

\begin{figure}
\begin{center}
\includegraphics[scale=0.32]{Figures/Ex3_MMV70_30}
\includegraphics[scale=0.32]{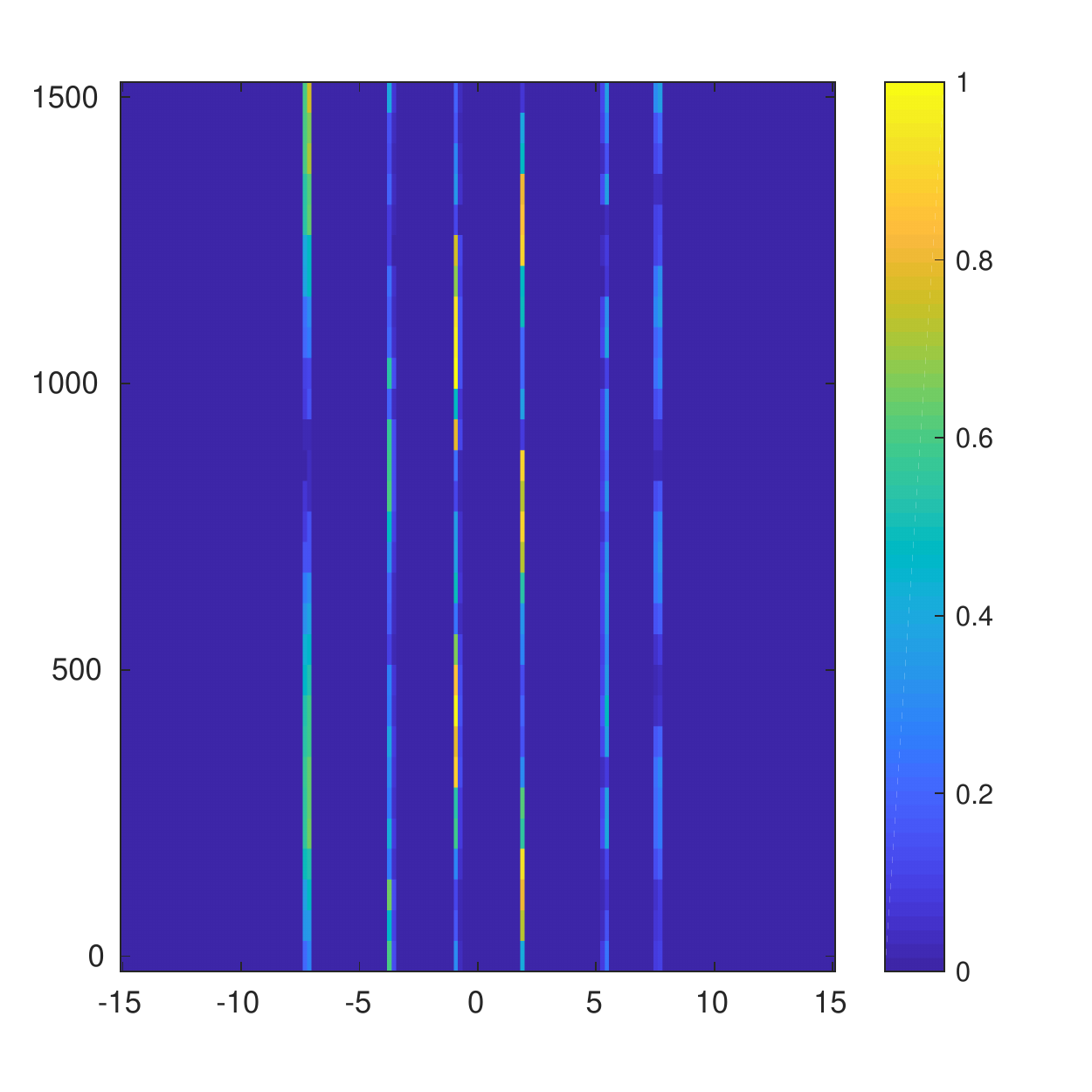}
\end{center}
\vspace{-0.12in}
\caption{MMV reconstructions with noiseless data (left) and noisy data (right).
The noise is additive, complex Gaussian, with mean zero independent entries and standard deviation 
$\sigma = 10\%$ of the largest entry in $\bD$. The axes are as in Figure \ref{fig:Image1}.}
\label{fig:Image3Noise}
\end{figure}

\section{SAR imaging with polarization diverse measurements}
\label{sect:applic2}
In this section we describe  briefly the application of  SAR imaging with polarization.  
We begin in section \ref{sect:Pol1} with the derivation of the data model \eqref{eq:th1}
used in the MMV formulation and then show numerical results in section 
\ref{sect:Pol2}.
\subsection{Data model}
\label{sect:Pol1}
Consider  a collection of $|\cS|$
penetrable scatterers, with volume smaller than $\ola^3$ by a
factor $\alpha \ll 1$, so that the scattered electric field at the
SAR platform can be modeled by \cite{ammari2007music}
\begin{equation}
\boldsymbol{\mathscr{E}}(\br_j,t; \obr, {\itbf f}) = \int \frac{d
  \om}{2 \pi} e^{-i \om t} i k^3(\om) \sqrt{\frac{\mu}{\epsilon}}
\sum_{q \in \cS} \hat{\boldsymbol{\mathscr{G}}}(\om,\br_j,\by_q)
\brho_q(\obr) \hat{\boldsymbol{\mathscr{G}}}(\om,\by_q,\br_j) \hat{\itbf
  f}(\om) + O(\alpha^4),
\label{eq:Pol1}
\end{equation}
where $\ola$ is the central wavelength and $\mu$ and $\epsilon$ are the magnetic permeability and the
electric permittivity in the medium. These define the wave speed $c =
1/\sqrt{\mu \epsilon}$ and the wavenumber $k(\om) = \om/c$. The
scatterers are represented in \eqref{eq:Pol1} by their center location
$\by_q$ and their reflectivity tensor assumed constant over the
sub-aperture centered at $\obr$,
\begin{equation}
\brho_q(\obr) = \alpha^3 \left(\frac{\epsilon_q}{\epsilon} -1 \right)
     {\itbf M}_q(\obr),
\label{eq:Pol2}
\end{equation}
where $\epsilon_q$ is the electric permittivity in the scatterer and
${\itbf M}_q$ is its $\alpha$--independent polarization tensor. We
refer to \cite{ammari2013mathematical} for details on ${\itbf M}_q$,
which depends on the shape of the scatterer. Here we assume that it is
a real valued $3 \times 3$ symmetric matrix. Since we consider a fixed central frequency $\oom$, 
we suppress in the notation the dependence of $\brho_q$ on $\oom$. We also neglect the
variation of the magnetic permeability in the scatterer, although this
can be taken into account, as shown in \cite{ammari2013mathematical}.

The wave propagation from the radar platform to the scatterers and
back is modeled in \eqref{eq:Pol1} by the dyadic Green's tensor
\begin{equation}
\hat{\boldsymbol{\mathscr{G}}}(\om,\br,\by) = \left({\itbf I} + \frac{\nabla
  \nabla^T}{k^2(\om)}\right) \frac{\exp[i k(\om) |\br-\by|]}{4 \pi
  |\br-\by|},
\label{eq:Pol4}
\end{equation}
where ${\itbf I}$ is the $3 \times 3$ identity matrix.  The wave
excitation is modeled by the vector $\hat{\itbf f}$. To avoid a lengthy discussion\footnote{In fact, only the transverse 
components of the electric field, in the plane orthogonal to the range direction $\obr - \oby$, play a role in the end, as discussed at the 
end of this section.} suppose
that the radar  emits and receives all possible
polarizations, so that we have access to the $3 \times 3$ frequency
dependent data matrix
\begin{equation}
\hat{\boldsymbol{\mathscr{D}}}(\br_j,\om;\obr) \approx \sum_{q = 1}^{N_\by}
\hat{\boldsymbol{\mathscr{G}}}(\om,\br_j,\by_q) \brho_q(\obr)
\hat{\boldsymbol{\mathscr{G}}}(\om,\by_q,\br_j),
\label{eq:Pol3}
\end{equation}
with the approximation due to the neglected $O(\alpha^4)$
residual. Here we sum over all the $N_\by$
points in the imaging region, with the convention that $\brho_q = 0$
for $q \notin \cS$.

As in the previous section, we focus attention on imaging in the
cross-range direction. This is why it is sufficient to consider a single
frequency, equal to the central one $\oom$. The wave number at this
frequency is denoted by $\ok$, as in the previous section.

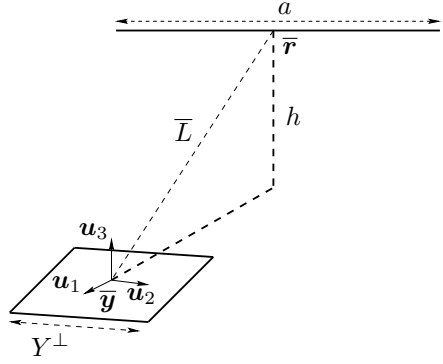
\begin{figure}[t]
\begin{center}
\includegraphics[width=0.35\textwidth]{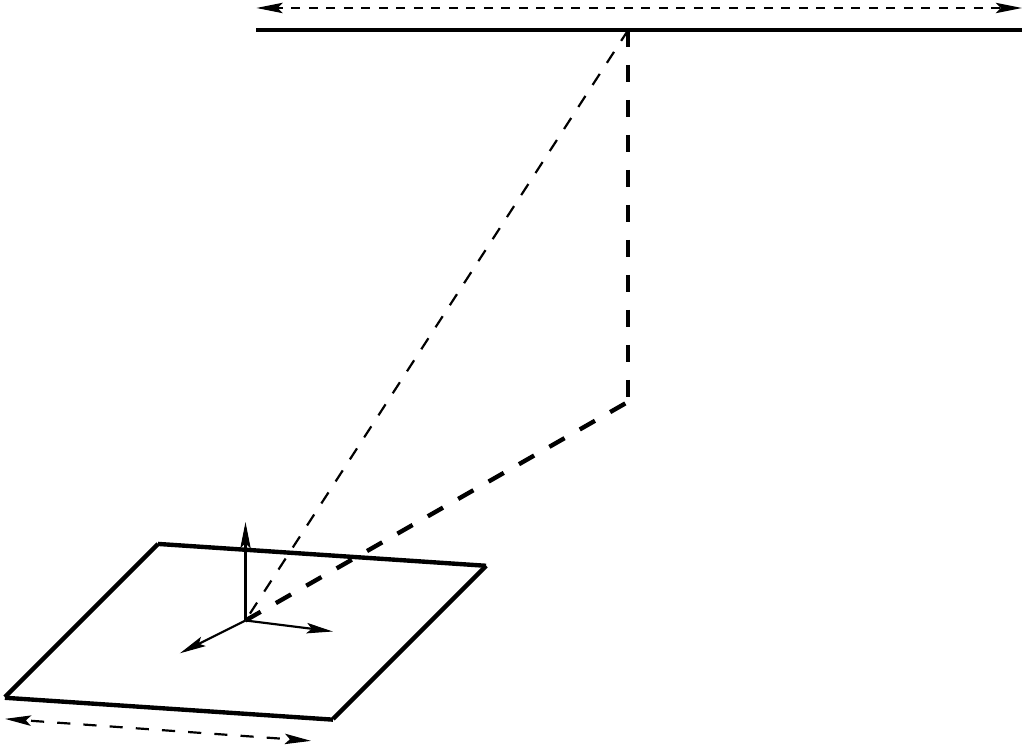}
\end{center}
\setlength{\unitlength}{3947sp}%
\begingroup\makeatletter\ifx\SetFigFont\undefined%
\gdef\SetFigFont#1#2#3#4#5{
  \fontfamily{#3}\fontseries{#4}\fontshape{#5}
  \fi\endgroup%
\begin{picture}(9770,1095)(1568,-1569)
\put(5150,750){\makebox(0,0)[lb]{\smash{{\SetFigFont{7}{8.4}{\familydefault}{\mddefault}{\updefault}{\color[rgb]{0,0,0}{\normalsize $\oL$}}%
}}}}
\put(5850,850){\makebox(0,0)[lb]{\smash{{\SetFigFont{7}{8.4}{\familydefault}{\mddefault}{\updefault}{\color[rgb]{0,0,0}{\normalsize $h$}}%
}}}}
\put(5800,1550){\makebox(0,0)[lb]{\smash{{\SetFigFont{7}{8.4}{\familydefault}{\mddefault}{\updefault}{\color[rgb]{0,0,0}{\normalsize $a$}}%
}}}}
\put(5830,1300){\makebox(0,0)[lb]{\smash{{\SetFigFont{7}{8.4}{\familydefault}{\mddefault}{\updefault}{\color[rgb]{0,0,0}{\normalsize $\obr$}}%
}}}}
\put(4250,-600){\makebox(0,0)[lb]{\smash{{\SetFigFont{7}{8.4}{\familydefault}{\mddefault}{\updefault}{\color[rgb]{0,0,0}{\normalsize $Y^\perp$}}%
}}}}
\put(4850,-250){\makebox(0,0)[lb]{\smash{{\SetFigFont{7}{8.4}{\familydefault}{\mddefault}{\updefault}{\color[rgb]{0,0,0}{\normalsize $\bu_2$}}%
}}}}
\put(4380,-170){\makebox(0,0)[lb]{\smash{{\SetFigFont{7}{8.4}{\familydefault}{\mddefault}{\updefault}{\color[rgb]{0,0,0}{\normalsize $\bu_1$}}%
}}}}
\put(4550,160){\makebox(0,0)[lb]{\smash{{\SetFigFont{7}{8.4}{\familydefault}{\mddefault}{\updefault}{\color[rgb]{0,0,0}{\normalsize $\bu_3$}}%
}}}}
\put(4680,-290){\makebox(0,0)[lb]{\smash{{\SetFigFont{7}{8.4}{\familydefault}{\mddefault}{\updefault}{\color[rgb]{0,0,0}{\normalsize $\oby$}}%
}}}}
\end{picture}%
\vspace{-0.8in}
\caption{Geometry of the data acquisition. The radar platform flies at
  elevation $h$ from the plane surface containing the imaging region
  $\Omega$, centered at $\oby$. The distance $\oL$ from the center
  $\obr$ of the aperture to $\oby$ is order $h$. The drawing is not up
  to scale, as the aperture $a$ and side $Y^\perp$ of the imaging region are
  much smaller than $\oL$.  }
\label{fig:PolSETUP}
\end{figure}

The sub-aperture centered at $\obr$ is
linear, of length $a$,  like before, and we assume for simplicity that it is at
constant altitude $h$, as shown in Figure \ref{fig:PolSETUP}. We let
$\bu_3$ be the unit vector in the vertical direction, and introduce
the unit vector $\bu_1 = \btau \times \bu_3$, where $\btau$ is the
unit tangent to the aperture, orthogonal to $\bu_3$.  The imaging region 
$\Omega$ is in the 
plane spanned by $\bu_1$ and $\btau$. We are interested in its cross-section 
in the  direction  of the aperture, which is the cross-range 
interval centered at $\oby$, of length $Y^\perp$. 

In the system of coordinates
with center at $\oby$ and orthonormal basis $\{\bu_j\}_{1\le j \le
  3}$ with $\bu_2 = \btau$, we have
$
\br = r_1 \bu_1 + r_2 \bu_2 + h \bu_3$ and  $\by =  y_2 \bu_2,$
for all $\br$ in the aperture and $\by$ in the cross-range 
imaging interval.  We also represent
the symmetric $3\times 3$ reflectivity tensor $\brho_q(\obr)$ by the
$1\times 6$ row-vector formed with the entries in its upper-tridiagonal
part 
\[
\brho_{q\rightarrow} = (\rho_{q,11}, \rho_{q,22},
\rho_{q,33},\rho_{q,12},\rho_{q,13}, \rho_{q,23}), \quad \rho_{q,jl} =
\bu_j^T \brho_q \bu_l.
\]

The scaling regime is as in the previous section, with length scales
ordered as $\ola \ll Y^\perp \lesssim a \ll h$, satisfying 
$ \oL = |\obr|= O(h)$ and $ |r_j| = O(\oL)$, for $ j = 1, 2.$
The Green tensor \eqref{eq:Pol4} has the following approximation in this regime
\begin{align}
\hat{\boldsymbol{\mathscr{G}}}(\om,\br_j,\by_q) 
&\approx \frac{\exp[i \ok |\br_j-\by_q|]}{4 \pi
  \oL} \begin{pmatrix} 1 - \eta_1^2 & -\eta_1 \eta_2 & -\eta_1 \beta \\ -\eta_1\eta_2 & 1 -\eta_2^2& -\eta_2 \beta
  \\- \eta_1 \beta & -\eta_2 \beta & 1-\beta^2 \end{pmatrix},\quad \eta_j =
\overline{r}_j/\oL, ~ ~j = 1, 2,~ ~ \beta = h/\oL.
\label{eq:Pol6p}
\end{align}
Substituting it in \eqref{eq:Pol3}, and representing the symmetric
matrix $(4 \pi \oL)^2/\sqrt{N_r}
\hat{\boldsymbol{\mathscr{D}}}(\br_j,\oom;\obr)$ by the $1 \times 6$ row
vector formed with the entries in its upper triangular part, we obtain
the data model
\begin{align}
\bd_{j\rightarrow}(\obr) = \sum_{q=1}^{N_\by} \frac{\exp[2 i \ok
    |\br_j-\by_q|]}{\sqrt{N_r}}\brho_{q \rightarrow}(\obr)
\boldsymbol{\Gamma}(\obr), \quad j = 1, \ldots, N_r,
\label{eq:Pol8}
\end{align}
with  $ \obr= \overline{r}_1 \bu_1 + \overline{r}_2 \bu_2 + h \bu_3$  and constant matrix  $\boldsymbol{\Gamma} (\obr)$ given in 
Appendix \ref{ap:Gamma}.
This is a linear system of form \eqref{eq:th1},
for $N_v = 6$, data matrix $\bD \in \CC^{N_r\times 6}$ with rows
$\bd_{j\rightarrow}$, unknown matrix $\bX\in \CC^{N_{\by} \times 6}$
with rows
\begin{equation}
\bx_{q \rightarrow} = \brho_{q \rightarrow} \boldsymbol{\Gamma},
\label{eq:defXrow}
\end{equation}
and sensing matrix $\bG$ with normalized
columns $\bg_q = \frac{1}{\sqrt{N_r}} \Big(\exp[2 i \ok |\br_1-\by_q|],
\ldots, \exp[2i \ok |\br_{N_r}-\by_q|] \Big)^T$.  

The system \eqref{eq:Pol8} is for a single
sub-aperture. More sub-apertures, centered at $\obr_v$, can be taken
into account as explained in the previous section, with the only
difference being that instead of having a scalar unknown, we now
have the unknown $1\times 6$ row vector
$\brho_q(\obr_v) \boldsymbol{\Gamma}(\obr_v).$ The linear system that
fuses the data from all the sub-apertures is obtained as in section
\ref{sect:applic1.2}, and the unknown matrix $\bX$ has six times more
columns than in the acoustic case.

Note that the approximation  \eqref{eq:Pol6p} of the Green's tensor $\hat{\boldsymbol{\mathscr{G}}}(\om,\br_j,\by_q)$ 
for the sub-aperture centered at $\obr$ has 
the one dimensional null space $\mbox{span}\{\obr\}$. This implies that 
the matrix $\boldsymbol{\Gamma} (\obr)$ is also singular, so we cannot determine uniquely
the reflectivity vectors $\brho_{q\rightarrow}$ from equation \eqref{eq:defXrow}. 
To be more explicit, we can represent the reflectivity tensor $\brho_q$ in \eqref{eq:Pol3} in the 
sub-aperture dependent orthonormal basis 
$\{\bv_j\}_{j=1,2,3}$ of eigenvectors of the  matrix in  \eqref{eq:Pol6p},  with $\bv_3 = \obr/|\obr|$. 
Then, we obtain that the components $\{\bv_j^T \brho_q \bv_3\}_{j=1,2,3}$ play no role in the data model \eqref{eq:Pol3}, so 
we can only estimate $(\bv_j^T \brho_q \bv_l)_{j,l = 1,2}$.   This 
ambiguity is due to  the scaling relation  $a/|\obr| \ll 1$ and it  implies that only the transverse components of the electric 
field  are needed in imaging, as the longitudinal component along $\bv_3$ adds no information.  
If the reflectivity tensor does not change over directions, or it changes slowly, then the ambiguity can be overcome by 
taking into consideration the multiple sub-apertures, because $\obr$ changes orientation from one sub-aperture to another.

\subsection{Numerical results}
\label{sect:Pol2}
\begin{figure}[t]
\begin{center}
\includegraphics[scale=0.5]{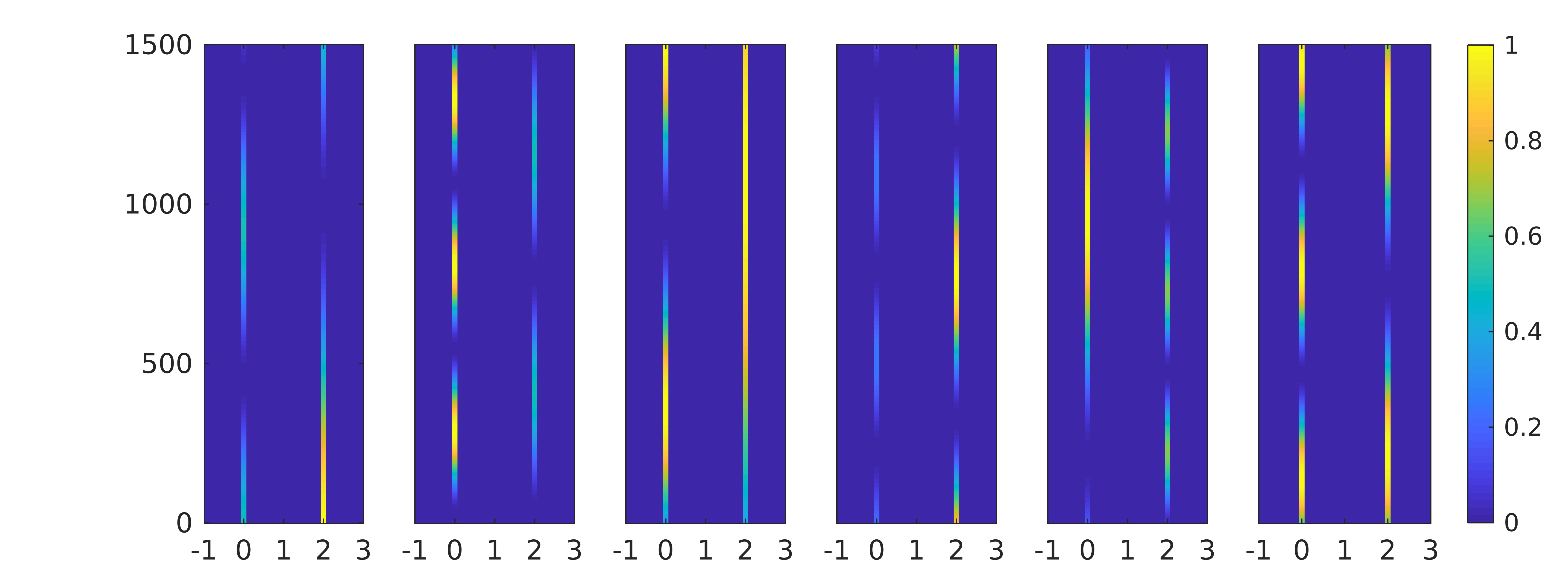}
\\ 
\includegraphics[scale=0.5]{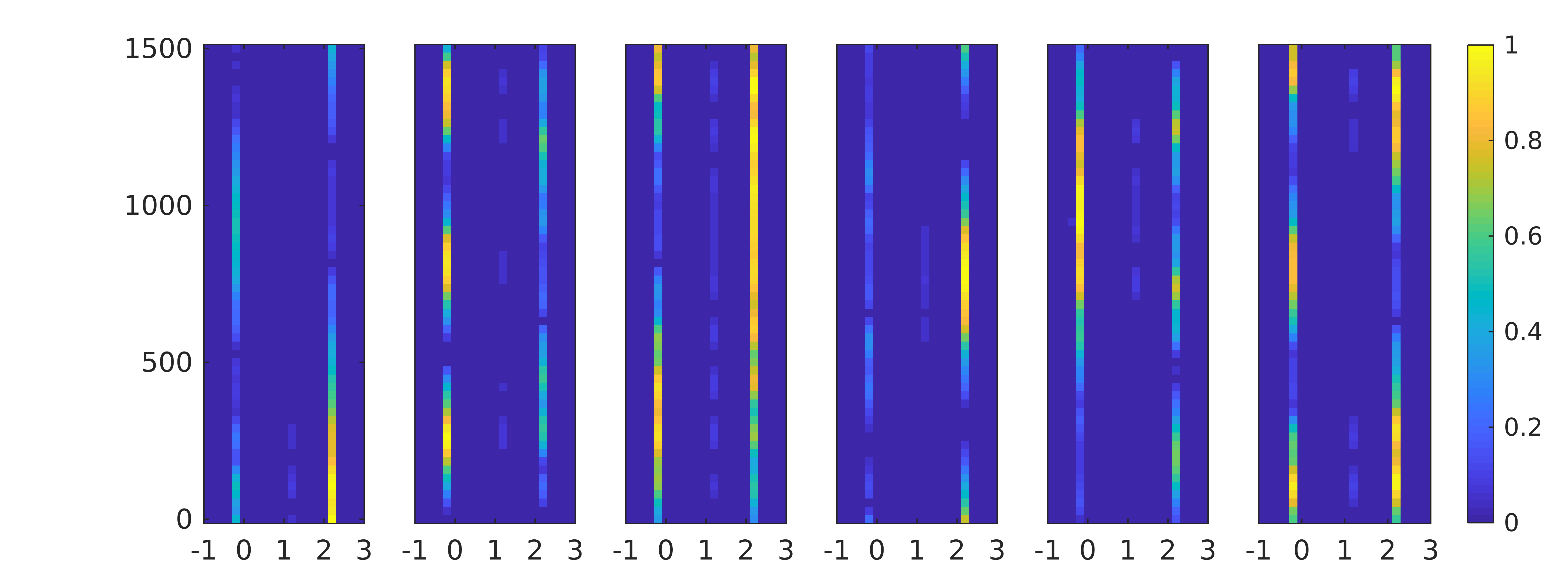}
\end{center}
\vspace{-0.1in}
\caption{Top line: From left to right we display all six components of
  the row vector $\bx_{q\rightarrow}(\obr)$ defined in
  \eqref{eq:defXrow}, as a function of location along the aperture
  (the ordinate in meters) and cross-range location indexed by $q$ in
  the imaging region (the abscissa, in units of $\ola \, 
  \oL_o/A$). Bottom line: The MMV reconstruction.
  }
\label{fig:Image4}
\end{figure}

\begin{figure}[h]
\begin{center}
\includegraphics[scale=0.36]{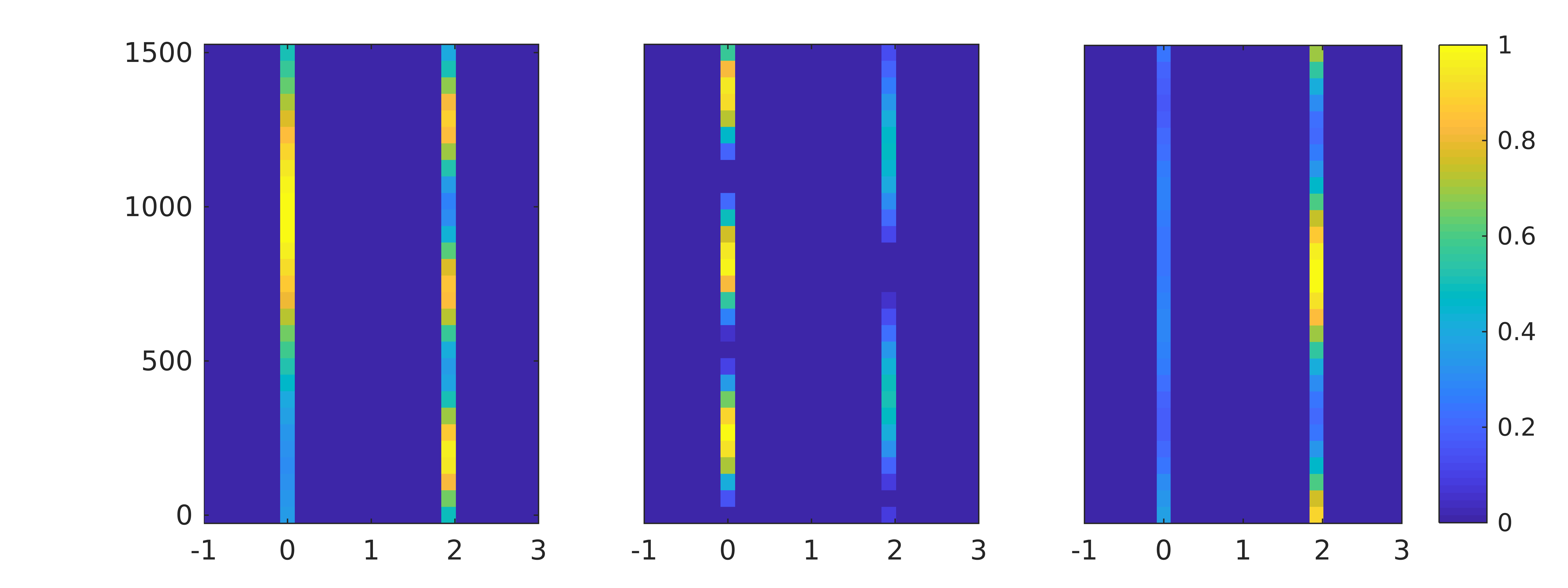}
\\ 
\includegraphics[scale=0.36]{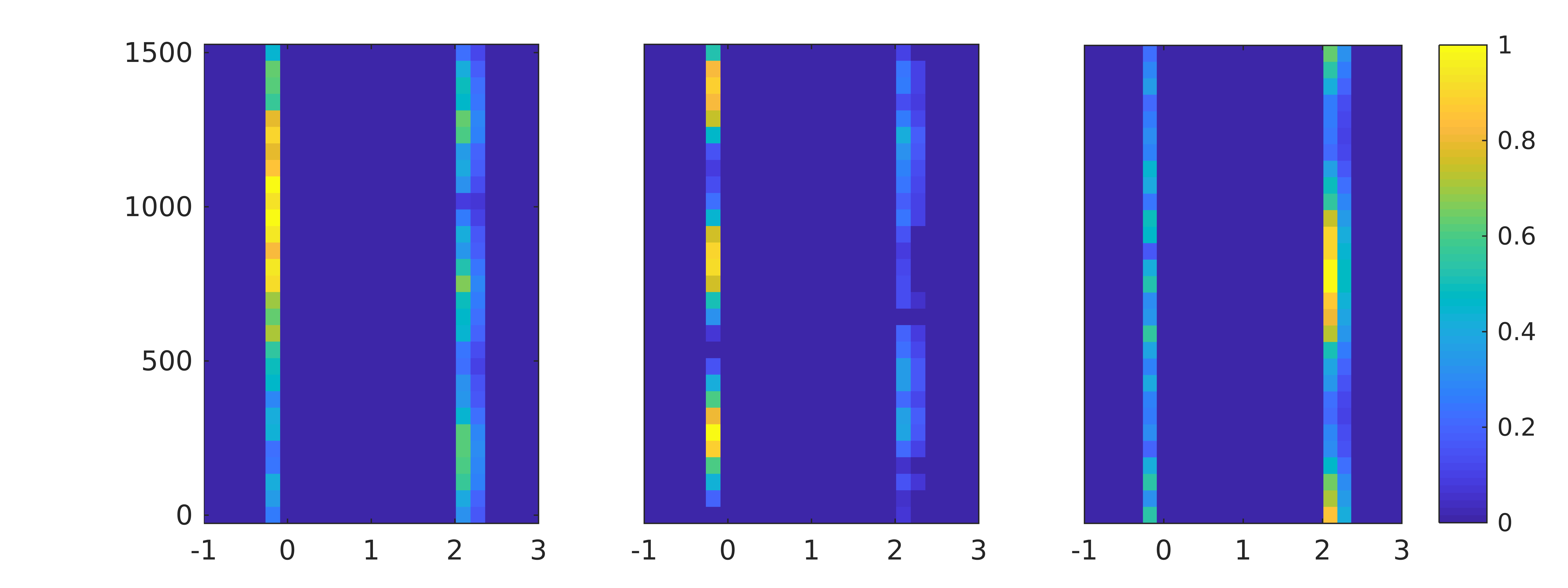}
\end{center}
\vspace{-0.1in}
\caption{Top line: The components $\bv_j^T \brho_q \bv_l$ 
of the reflectivity matrix, for $j=l=1$ (left plot), $j= l = 2$ (middle plot) and 
$j=1, l = 2$ (right plot). The orthonormal basis $(\bv_j)_{j=1,2,3}$ depends on the 
center location $\obr$ of the sub-aperture 
  (the ordinate in meters).  The abscissa is the cross-range location indexed by $q$,  in units of $\ola \, 
  \oL_o/A$. Bottom line: The reconstruction.
  }
\label{fig:Image5}
\end{figure}
The setup for the numerical results is the same as in
section \ref{sect:applic1.4}. The data are generated using the
single scattering model \eqref{eq:Pol1}, for a reflectivity function
that changes with the direction of illumination and  is
supported at two points at distance of order $\ola \, \oL_o/A$, where
$\oL_o = |\obr_o-\oby|$. 

We display in Figure \ref{fig:Image4}  the six entries of the row
vectors $\bx_{q \rightarrow}$ defined in \eqref{eq:defXrow}, 
as
$\obr$ varies in the large aperture, and for points in $\Omega$ indexed by
$q$, separated by distances  $\ola \, \oL_o/A$ in cross-range. The plots in the bottom
line of Figure \ref{fig:Image4} show that the MMV method gives good
estimates of these row vectors.
 
In Figure \ref{fig:Image5}  we display the components $(\bv_j^T \brho_q \bv_l)_{j,l=1,2}$  of the reflectivity matrix $\brho_{q}$ and its reconstruction, 
for each sub-aperture centered at $\obr$. As in the note at the end of the previous section,  
we let $\{\bv_j\}_{j=1,2,3}$ be the orthonormal basis of eigenvectors of the approximation \eqref{eq:Pol6p} of the Green's tensor,
with $\bv_3$ along $\obr$. The reconstruction displayed
in Figure \ref{fig:Image5} is calculated as follows: With the estimated vectors $\bx_{q \rightarrow}$ displayed in 
Figure \ref{fig:Image4} we calculate the minimum $\ell_2$ norm solution of \eqref{eq:defXrow}, using the truncated SVD of the 
singular matrix $\boldsymbol{\Gamma} (\obr)$. This corresponds to setting to zero the components $\bv_j^T \brho_q \bv_l$ 
of the estimated $\brho_q$, for either $j$ or $l$ equal to $3$. The other components are displayed in the figure, and they are well 
reconstructed.

\section{Proofs}
\label{sect:proofs}
Here we prove the results stated in section \ref{sect:theory2}. We
begin with a Lemma, in section \ref{sect:proof0}, which we then use in
sections \ref{sect:proof1} and \ref{sect:proof2} to prove Theorems
\ref{thm:1} and \ref{thm:2}.  Proposition \ref{prop:1} is proved in
section \ref{sect:proof3} and the results for the clusters are proved
in section \ref{sect:proof4}.
\subsection{A basic lemma}
\label{sect:proof0}

Let us denote by $\wh{\bX}$ the matrix obtained by normalizing the
nonzero rows in $\bX$, the unknown in the inverse problem,
$  \wh{\bx}_{q \rightarrow} = \frac{\bx_{q \rightarrow}}{\|\bx_{q
      \rightarrow}\|_2}$,  for $q \in \cS$. Introduce the linear operator 
      \begin{equation}
 \cL: \CC^{N_{\br} \times N_v}
\to \CC, \quad  \cL(\bV) = \mbox{tr}\Big[ (\bG \wh{\bX})^\star \bV \Big], \quad \forall \,
  \bV \in \CC^{N_{\br} \times N_v},
  \label{eq:Pf2}
\end{equation}
where $\mbox{tr}[\cdot]$ denotes the trace. We have the following result:

\vspace{0.05in}
\begin{lemma}
  \label{lem:1}
The linear operator $\cL$ defined in \eqref{eq:Pf2} satisfies the
inequality
\begin{equation}
  \big|\cL(\bV)\big| \le \|(\bG^\star \bV)_{\cS\rightarrow} \|_{1,2},
  \label{eq:Pf6}
\end{equation}
for any $\bV \in \CC^{N_{\br} \times N_v}$. The matrix $\bX$ satisfies
the inequality
\begin{equation}
  \|\bX\|_{1,2} \big(1 - \cI_{N_v}\big) \le \big| \cL(\bG \bX)\big|,
  \label{eq:Pf3}
\end{equation}
and with $r$, $\bX^{\ep,r}$ and $\bE^{\ep,r}$ defined as in Theorem
\ref{thm:1}, we have
\begin{align}
  \big|\cL(\bG \bX^{\ep,r})\big| &\le \big(1 + \cI_{N_v}\big) \| \bX^{\ep,r}\|_{1,2},
  \label{eq:Pf4} \\
  \big|\cL(\bG \bE^{\ep,r})\big| &\le \big(1 -r + \cI_{N_v}\big) \|
  \bE^{\ep,r}\|_{1,2}.
  \label{eq:Pf5}
\end{align}
\end{lemma}
\begin{proof} We start with  definition \eqref{eq:Pf2}, and use
the invariance of the trace under cyclic permutations, and the row
support $\cS$ of $\bX$, to obtain 
\begin{align*}
  \mathfrak{L}(\bG \bX) &= \mbox{tr}\Big[ (\bG \wh{\bX})^\star \bG \bX
    \Big] =  \mbox{tr}\Big[ \bX \wh \bX^\star \bG^\star \bG \Big] =
  \sum_{j,q \in \cS} (\bX \wh \bX^\star)_{j,q}  (\bG^\star
  \bG)_{q,j} = \sum_{j,q \in \cS} \lin \bx_{j \rightarrow},
  \wh{\bx}_{q \rightarrow} \rin \lin \bg_q, \bg_j \rin.
\end{align*}
We rewrite this further with the normalization condition
\eqref{eq:normg} and definition \eqref{eq:th13}, and use the
triangle inequality to obtain the bound 
\begin{align*}
  \big|\mathfrak{L}(\bG \bX)\big| &= \Big| \sum_{q \in \cS} \Big[\lin
    \bx_{q \rightarrow}, \wh{\bx}_{q \rightarrow} \rin \lin \bg_q,
    \bg_q \rin +\sum_{j \in \cS\setminus \{q\}}\lin \bx_{j
      \rightarrow}, \wh{\bx}_{q \rightarrow} \rin \lin \bg_q, \bg_j
    \rin \Big]\Big| \\ &= \Big| \sum_{q \in \cS}
  \|\bx_{q\rightarrow}\|_2 \Big[ 1 + \sum_{j \in \cS \setminus \{q\}
    }\mu(\bx_{j \rightarrow}, {\bx}_{q \rightarrow}) \mu(\bg_q,\bg_j)
    \Big] \Big| \\ & \ge \sum_{q \in \cS} \|\bx_{q\rightarrow}\|_2
  \Big[1 - \sum_{j \in \cS \setminus \{q\}}|\mu(\bx_{j \rightarrow},
    {\bx}_{q \rightarrow})| |\mu(\bg_q,\bg_j)| \Big] \Big| \\ &\ge
  \sum_{q \in \cS} \|\bx_{q\rightarrow}\|_2 (1- \cI_{N_v}).
\end{align*}
The result \eqref{eq:Pf3} follows from definition
of the matrix norm $\|\bX\|_{1,2}$.

Similarly, 
\begin{align*}
  \mathfrak{L}(\bG \bX^{\ep,r}) = \sum_{j \in \cS^\ep} \sum_{q \in \cS}
  \lin \bx_{j\rightarrow}^{\ep,r}, \wh \bx_{q
    \rightarrow} \rin \lin \bg_q,\bg_j \rin = \sum_{j \in \cS^\ep}
  \|\bx_{j\rightarrow}^{\ep,r}\|_2\sum_{q \in \cS} 
  \mu(\bx_{j\rightarrow}^{\ep,r}, \wh \bx_{q
    \rightarrow}) \mu(\bg_q,\bg_j),
\end{align*}
where $\bx_{j\rightarrow}^{\ep,r}$ denotes the
$j$--th row of $\bX^{\ep,r}$. Using the decomposition \eqref{eq:th24} of
the row support $\cS^\ep$ of $\bX^{\ep,r}$, 
\begin{align*}
  \big|\mathfrak{L}(\bG \bX^{\ep,r})\big| &= \Big|\sum_{i \in \cS}
  \sum_{j \in \cS^\ep_i}
  \|\bx_{j\rightarrow}^{\ep,r}\|_2 \sum_{q \in \cS}
  \mu(\bx_{j\rightarrow}^{\ep,r}, \wh \bx_{q
    \rightarrow}) \mu(\bg_q,\bg_j)\Big| \\ &= \Big|\sum_{i \in \cS}
  \sum_{j \in \cS^\ep_i}
  \|\bx_{j\rightarrow}^{\ep,r}\|_2
  \Big[\mu(\bx_{j\rightarrow}^{\ep,r}, \wh \bx_{i
      \rightarrow}) \mu(\bg_i,\bg_j) + \sum_{q \in \cS \setminus\{i\}}
    \mu(\bx_{j\rightarrow}^{\ep,r}, \wh \bx_{q
    \rightarrow}) \mu(\bg_q,\bg_j)\Big]\Big|.
\end{align*}
By the construction in \eqref{eq:th24}, for any $j \in
\cS^\ep_i$, the index $n(j)\in \cS$ of the nearest point to
$\by_j$ is $n(j) = i$, so the sum in $q$ is over the set $\cS
\setminus\{n(j)\}$.  Using the triangle inequality and the definition
\eqref{eq:th13} of $\cI_{N_v}$, we get
\begin{align*}
   \big|\mathfrak{L}(\bG \bX^{\ep,r})\big| &\le \sum_{i \in \cS}
   \sum_{j \in \cS^\ep_i}
   \|\bx_{j\rightarrow}^{\ep,r}\|_2
   \Big[|\mu(\bx_{j\rightarrow}^{\ep,r}, \wh \bx_{i
       \rightarrow}) \mu(\bg_i,\bg_j)| + \sum_{q \in \cS
       \setminus\{n(j)\}}
     |\mu(\bx_{j\rightarrow}^{\ep,r}, \wh \bx_{q
       \rightarrow})| \mu(\bg_q,\bg_j)|\Big] \\
   & \le \sum_{i \in \cS}
   \sum_{j \in \cS^\ep_i}
   \|\bx_{j\rightarrow}^{\ep,r}\|_2
   (1 + \cI_{N_v}) = \sum_{j \in \cS^\ep} \|\bx_{j\rightarrow}^{\ep,r}\|_2
   (1 + \cI_{N_v}).
\end{align*}
Since $\cS^\ep$ is the row support of $\bX^{\ep,r}$, we can extend the sum
to $1 \le j \le N_{\by}$, 
and the result \eqref{eq:Pf4} follows from the definition of the
$\|\cdot \|_{1,2}$ norm.

To prove \eqref{eq:Pf5}, recall that $\bE^{\ep,r}$ is supported by definition
 in the set 
$
\cB_r^c(\cS) = \{1, \ldots, N_{\by}\} \setminus \cB_r(\cS).
$
Then, if we denote by $ \be_{j\rightarrow}^{\ep,r}$ the rows
of $\bE^{\ep,r}$, we have 
\begin{align*}
  \mathfrak{L}(\bG \bE^{\ep,r}) &= \sum_{j \in \cB_r^c(\cS)} \sum_{q \in
    \cS} \lin \be_{j\rightarrow}^{\ep,r}, \wh \bx_{q
    \rightarrow} \rin \lin \bg_q,\bg_j \rin \\ &= \sum_{j \in
    \cB_r^c(\cS)}
  \|\be_{j\rightarrow}^{\ep,r}\|_2\sum_{q \in \cS}
  \mu(\be_{j\rightarrow}^{\ep,r}, \wh \bx_{q
    \rightarrow}) \mu(\bg_q,\bg_j) \\ &= \sum_{j \in \cB_r^c(\cS)}
  \|\be_{j\rightarrow}^{\ep,r}\|_2 \Big[
    \mu(\be_{j\rightarrow}^{\ep,r}, \wh \bx_{n(j)
      \rightarrow}) \mu(\bg_{n(j)},\bg_j) + \sum_{q \in
      \cS\setminus\{n(j)\}}
    \mu(\be_{j\rightarrow}^{\ep,r}, \wh \bx_{q
      \rightarrow}) \mu(\bg_q,\bg_j)\Big].
\end{align*}
Taking the absolute value and using the triangle inequality and 
definition \eqref{eq:th13} of $\cI_{N_v}$, we obtain the bound
\begin{align*}
  \big|\mathfrak{L}(\bG \bE^{\ep,r})\big| \le \sum_{j \in \cB_r^c(\cS)}
  \|\be_{j\rightarrow}^{\ep,r}\|_2 \Big[
    |\mu(\be_{j\rightarrow}^{\ep,r}, \wh \bx_{n(j)
      \rightarrow})| |\mu(\bg_{n(j)},\bg_j)| + \cI_{N_v}\Big].
\end{align*}
But $|\mu(\be_{j\rightarrow}^{\ep,r}, \wh \bx_{n(j)
  \rightarrow})| \le 1$ and $|\mu(\bg_{n(j)},\bg_j)| = 1 -
\cd(j,n(j))$, with $\cd(j,n(j)) \ge r$ for any $j \in
\cB_r^{c}(\cS)$, so the bound becomes
\begin{align*}
  \big|\mathfrak{L}(\bG \bE^{\ep,r})\big| \le \sum_{j \in \cB_r^c(\cS)}
  \|\be_{j\rightarrow}^{\ep,r}\|_2 (1 - r + \cI_{N_v}).
\end{align*}
We can extend the sum to $1 \le j \le N_{\by}$ because $\bE^{\ep,r}$ is
supported in $\cB_r^{c}(\cS)$, and the result \eqref{eq:Pf5} follows
from the definition of the $\|\cdot \|_{1,2}$ norm.

Finally, for any $\bV \in \CC^{N_{\br} \times N_v}$, we obtain using the
invariance of the trace to cyclic permutations that  
\begin{align*}
\mathfrak{L}(\bV) &= \mbox{tr}\Big[(\bG \wh \bX)^\star \bV \Big]
=\mbox{tr}\Big[ \bG^\star \bV \wh \bX^\star\Big] 
= \sum_{j = 1}^{N_{\by}} \lin \big(\bG^\star \bV)_{j \rightarrow}, \wh
\bx_{j \rightarrow} \rin = \sum_{j \in \cS}  \lin \big(\bG^\star \bV)_{j \rightarrow}, \wh
\bx_{j \rightarrow} \rin,
\end{align*}
where the last equality is because $\bX$ is row supported in $\cS$.
Taking the absolute value and using the triangle and Cauchy-Schwartz
inequalities we get
\begin{align*}
  \big|\mathfrak{L}(\bV)\big| &\le \sum_{j \in \cS} \big|\lin
  \big(\bG^\star \bV)_{j \rightarrow}, \wh \bx_{j \rightarrow}
  \rin\big| \le \sum_{j \in \cS} \|\big(\bG^\star \bV)_{j \rightarrow}\|_2 = \|\big(\bG^\star \bV)_{\cS \rightarrow}\|_{1,2}.
\end{align*}
This is the result \eqref{eq:Pf6} in the lemma.
\end{proof}

\subsection{Proof of Theorem \ref{thm:1}}
\label{sect:proof1}
The bound \eqref{eq:th19p} follows from  the definition
of $\bW^\ep$ and the triangle inequality,
\begin{align*}
  \|\bW^\ep\|_F &= \|\bD_{\bW} - \bG \bX^\ep - \bW \|_{F} 
  \le \|\bD_{\bW} - \bG \bX^\ep\|_{F} + \|\bW \|_{F} 
  \le 2\ep,
\end{align*}
where we used the assumption \eqref{eq:th12} and that $\bX^\ep$ is the
minimizer of \eqref{eq:th11}.

Using again the definition of $\bW^\ep$ and the linearity of the operator
\eqref{eq:Pf2}, we write
\begin{align*}
  \cL(\bG \bX) + \cL(\bW^\ep) = \cL(\bG \bX + \bW^\ep) = \cL(\bG \bX^\ep) =
  \cL(\bG \bX^{\ep,r}) + \cL(\bG \bE^{\ep,r}),
\end{align*}
where the last equality is by the decomposition \eqref{eq:th18}. The result
\eqref{eq:Pf3} in Lemma \ref{lem:1} gives 
\begin{align*}
  \|\bX\|_{1,2} ( 1 - \cI_{N_v}) \le \big| \cL(\bG \bX)|  =
  \big| \cL(\bG \bX^{\ep,r}) + \cL(\bG \bE^{\ep,r}) - \cL(\bW^\ep)\big|,
\end{align*}
and using the triangle inequality and the estimates \eqref{eq:Pf6},
\eqref{eq:Pf4} and \eqref{eq:Pf5}, we get
\begin{align}
  \|\bX\|_{1,2} ( 1 - \cI_{N_v}) &\le \big| \cL(\bG \bX^{\ep,r})\big| +
  \big|\cL(\bG \bE^{\ep,r})| + \big|\cL(\bW^\ep)\big| \nonumber \\ & \le (1
  + \cI_{N_v}) \|\bX^{\ep,r}\|_{1,2} + (1-r+\cI_{N_v}) \|\bE^{\ep,r}\|_{1,2}
  + \|(\bG^\star \bW^\ep)_{\cS \rightarrow}\|_{1,2}. \label{eq:Pf7}
\end{align}
Note that by \eqref{eq:th10} and \eqref{eq:th12},
\[
\|\bG \bX - \bD_{\bW}\|_{F} = \|\bW\|_F < \ep,
\]
so since $\bX^\ep$ is the minimizer of \eqref{eq:th11}, we must have
$
  \|\bX^\ep\|_{1,2} \le \|\bX\|_{1,2}.$ We also obtain from the decomposition \eqref{eq:th18} of $\bX^\ep$ in
the matrices $\bX^{\ep,r}$ and $\bE^{\ep,r}$ with disjoint row support that
\[
  \|\bX^\ep\|_{1,2} = \|\bX^{\ep,r} + \bE^{\ep,r}\|_{1,2} =
  \|\bX^{\ep,r}\|_{1,2} + \|\bE^{\ep,r}\|_{1,2}.
 \]
Substituting in \eqref{eq:Pf7} we get
\begin{align*}
  \|\bX^\ep\|_{1,2} ( 1 - \cI_{N_v}) &\le (1 + \cI_{N_v})
  (\|\bX^\ep\|_{1,2} - \|\bE^{\ep,r}\|_{1,2}) + (1-r+\cI_{N_v})
  \|\bE^{\ep,r}\|_{1,2} + \|(\bG^\star \bW^\ep)_{\cS \rightarrow}\|_{1,2}
  \\ &= (1 + \cI_{N_v}) \|\bX^\ep\|_{1,2} - r \|\bE^{\ep,r}\|_{1,2} +
  \|(\bG^\star \bW^\ep)_{\cS \rightarrow}\|_{1,2} .
\end{align*}
We also have from  the definition of $\|\cdot \|_{1,2}$, the normalization 
of the columns of $\bG$ and  \eqref{eq:th19p}, that 
\begin{align}
\big\| \big(\bG^\star \bW^\ep\big)_{\cS \rightarrow}
    \big\|_{1,2} &= \sum_{j \in \cS} \|\bg_j^\star \bW^\ep\|_2 = 
\sum_{j \in \cS} \left[ \sum_{v=1}^{N_v} |\bg_j^\star {\bf w}_v^\ep|^2 \right]^{1/2} 
\hspace{-0.1in} \le \sum_{j \in \cS} \left[ \sum_{v=1}^{N_v}
  \|{\bf w}_v^\ep\|_2^2\right]^{1/2}\hspace{-0.1in} = |\cS| 
\|\bW^\ep\|_F \le 2 \ep |\cS|,
\label{eq:pessim}
\end{align}
where ${\bf w}_v^\ep$ are the columns of $\bW^\ep$.
The result \eqref{eq:th19} stated in the theorem follows. $\Box$

\subsection{Proof of Theorem \ref{thm:2}}
\label{sect:proof2}
Let us start with the definition of the matrices $\bW^\ep$, $\bX^{\ep,r}$
and $\bE^{\ep,r}$ given in Theorem \ref{thm:1}, and write
\begin{align*}
  \bG \bX^\ep = \bG (\bX^{\ep,r} + \bE^{\ep,r}) = \bG \bX + \bW^\ep.
\end{align*}
With the decomposition \eqref{eq:th21} of $\bX^{\ep,r}$, we get
\begin{align}
  \bG (\bX - \bxi^{\ep,r})  = \bG \bE^{\ep,r} + \bG \bEr^{\ep,r} - \bW^\ep,
  \label{eq:Pf10}
\end{align}
and we prove next the analogue of the result \eqref{eq:Pf7} for $\bX$
replaced by the matrix $\bX - \bxi^{\ep,r}$ and $\bX^{\ep,r}$ replaced by
$0$. Looking at the proof of \eqref{eq:Pf3} in section
\ref{sect:proof0}, we note that we only used that $\bX$ has row
support in $\cS$. The same holds for the matrix $\bX - \bxi^{\ep,r}$, so
we can write directly the analogue of \eqref{eq:Pf3}
\begin{align}
  \|\bX - \bxi^{\ep,r}\|_{1,2} (1 - \cI_{N_v}) \le \Big|\cL \Big(\bG (\bX -
  \bxi^{\ep,r})\Big)\Big|.
  \label{eq:Pf11}
\end{align}
The right hand side in this equation can be estimated using
\eqref{eq:Pf10} and the linearity of the operator $\cL$, 
\begin{align*}
  \Big|\cL\Big(\bG (\bX - \bxi^{\ep,r})\Big)\Big| = \Big|\cL(\bG
  \bE^{\ep,r}) + \cL ( \bG \bEr^{\ep,r} - \bW^\ep) \Big| 
  \le \Big|\cL(\bG
  \bE^{\ep,r}) \Big| + \Big|\cL ( \bG \bEr^{\ep,r} - \bW^\ep) \Big|.
\end{align*}
Substituting in \eqref{eq:Pf11} and using the estimates \eqref{eq:Pf5}
and \eqref{eq:Pf6}, with $\bV$ replaced by $\bG \bEr^{\ep,r} - \bW^\ep$,
we obtain
\begin{align*}
  \|\bX - \bxi^{\ep,r}\|_{1,2} (1 - \cI_{N_v}) \le (1-r + \cI_{N_v})
  \|\bE^{\ep,r}\|_{1,2} + \Big\| \Big(\bG^\star(\bG
  \bEr^{\ep,r} - \bW^\ep)\Big)_{\cS \rightarrow} \Big \|_{1,2}.
\end{align*}
But, by equation \eqref{eq:th23},
\[
\Big(\bG^\star\bG \bEr^{\ep,r})_{\cS \rightarrow} = \bG^\star_{\cS} \bG \bEr^{\ep,r} = 0,
\]
and the desired estimate is 
\begin{align}
  \|\bX - \bxi^{\ep,r}\|_{1,2} (1 - \cI_{N_v}) \le (1-r + \cI_{N_v})
  \|\bE^{\ep,r}\|_{1,2} + \Big\| \Big(\bG^\star \bW^\ep)\Big)_{\cS
    \rightarrow} \Big \|_{1,2},
  \label{eq:Pf12}
\end{align}
with the last term bounded as in \eqref{eq:pessim}.

Next, we substitute the bound \eqref{eq:th19} on the error term
$\bE^{\ep,r}$ in \eqref{eq:Pf12}, and obtain after simple algebraic
manipulations that
\begin{align}
  \|\bX - \bxi^{\ep,r}\|_{1,2} \le \frac{2 \cI_{N_v} (1-r + \cI_{N_v})}{r
    (1-\cI_{N_v})} \|\bX^\ep\|_{1,2} + \frac{(1+\cI_{N_v})}{r(1-\cI_{N_v})}
  \Big\| \Big(\bG^\star \bW^\ep)\Big)_{\cS
    \rightarrow} \Big \|_{1,2}.
  \label{eq:Pf13}
\end{align}
The assumption  $2 \cI_{N_v} < r < 1$ implies that
\[
1-r + \cI_{N_v} \le 1 - 2 \cI_{N_v} + \cI_{N_v} = 1 - \cI_{N_v} \quad \mbox{and} \quad \frac{1+\cI_{N_v}}{1-\cI_{N_v}} < \frac{1 + \cI_{N_v}}{1-r/2} < 2(1 +
\cI_{N_v}) < 3.
\]
Substituting in \eqref{eq:Pf13} we obtain the result \eqref{eq:th27}
of Theorem \ref{thm:2}.

It remains to prove the estimate \eqref{eq:th26}. We begin with the
identity
\[
\bxi^{\ep,r} - \overline{\bX^{\ep,r}} = \bX^{\ep,r} - \overline{\bX^{\ep,r}} -
\bEr^{\ep,r},
\]
and use equation \eqref{eq:th23} to conclude that
\begin{align*}
  \bG_\cS^\star \bG (\bxi^{\ep,r} - \overline{\bX^{\ep,r}}) =
  \bG_\cS^\star\bG (\bX^{\ep,r} - \overline{\bX^{\ep,r}}).
\end{align*}
By construction, both $\bxi^{\ep,r}$ and $\overline{\bX^{\ep,r}}$ are row
supported in $\cS$, so we can rewrite this equation as
\begin{align}
  \label{eq:Pf14}
  (\bxi^{\ep,r} -\overline{\bX^{\ep,r}})_{\cS \rightarrow} - ({\itbf
    I} - \bG_\cS^\star \bG_\cS) (\bxi^{\ep,r} -
  \overline{\bX^{\ep,r}})_{\cS \rightarrow} = \bG_\cS^\star\bG (\bX^{\ep,r}
  - \overline{\bX^{\ep,r}}),
\end{align}
where ${\itbf I}$ is the $|\cS| \times |\cS|$ identity matrix.
We now estimate  each term in this equation. 

For the right hand side in \eqref{eq:Pf14} we have
\begin{align}
  \|\bG_\cS^\star \bG (\bX^{\ep,r} - &\overline{\bX^{\ep,r}})\|_{1,1} =
  \sum_{q \in \cS} \sum_{v=1}^{N_v} \Big|\sum_{j = 1}^{N_{\by}}
  (\bG_\cS^\star \bG)_{q,j} (\bX^{\ep,r} - \overline{\bX^{\ep,r}})_{j,v}
  \Big| \nonumber\\&= \sum_{q \in \cS} \sum_{v=1}^{N_v} \Big|\sum_{j
    \in \cS \cup \cS^\ep} \mu(\bg_q,\bg_j) (\bX^{\ep,r} -
  \overline{\bX^{\ep,r}})_{j,v} \Big|\nonumber \\&= \sum_{q \in \cS}
  \sum_{v=1}^{N_v} \Big|\sum_{j \in \cS \cup \cS^\ep\setminus
    \cS^\ep_q} \hspace{-0.1in}\mu(\bg_q,\bg_j) (\bX^{\ep,r} -
  \overline{\bX^{\ep,r}})_{j,v} + \hspace{-0.15in} \sum_{j \in (\cS \cup \cS^\ep) \cap
   \cS^\ep_q}\hspace{-0.15in} \mu(\bg_q,\bg_j)
  (\bX^{\ep,r} - \overline{\bX^{\ep,r}})_{j,v}\Big|, \label{eq:Pf15}
\end{align}
where the first two equalities are by the definition of the norm and
of the matrix product, and the third equality uses the definition
\eqref{eq:th14} and the row support $\cS \cup \cS^\ep$ of $\bX^{\ep,r} -
\overline{\bX^{\ep,r}}$.  Now let us recall the definition
\eqref{eq:th25} of $\overline{\bX^{\ep,r}}$, and the
decomposition \eqref{eq:th24} of the support $\cS^\ep$ of $\bX^{\ep,r}$,
to obtain 
\[
\sum_{j \in (\cS \cup \cS^\ep) \cap
   \cS^\ep_q}\hspace{-0.1in} \mu(\bg_q,\bg_j)
  \bX^{\ep,r}_{j,v}  = \sum_{j \in 
   \cS^\ep_q}\hspace{-0.05in} \mu(\bg_q,\bg_j)
  \bX^{\ep,r}_{j,v} = \overline{\bX^{\ep,r}_{q,v}}
\quad \mbox{and} \quad 
\overline{\bX^{\ep,r}_{j,v}} = \overline{\bX^{\ep,r}_{q,v}} \, \delta_{j,q}, \quad
\forall \, j \in \cS^\ep_q,
\]
where $\delta_{j,q}$ is the Kronecker delta symbol. Since
$\mu(\bg_q,\bg_q) = 1$, we conclude that the second term in
\eqref{eq:Pf15} vanishes and the result becomes
\begin{align}
  \|\bG_\cS^\star \bG (\bX^{\ep,r} - \overline{\bX^{\ep,r}})\|_{1,1} &=
  \sum_{q \in \cS} \sum_{v=1}^{N_v} \Big|\sum_{j \in \cS \cup
    \cS^\ep\setminus \cS^\ep_q} \hspace{-0.1in}\mu(\bg_q,\bg_j)
  (\bX^{\ep,r} - \overline{\bX^{\ep,r}})_{j,v}\Big| \nonumber \\ &\le
  \sum_{v=1}^{N_v} \sum_{q\in \cS} \sum_{j \in \cS \cup
    \cS^\ep\setminus \cS^\ep_q} |\mu(\bg_q,\bg_j)| \Big|(\bX^{\ep,r}
  - \overline{\bX^{\ep,r}})_{j,v}\Big|.
  \label{eq:Pf16}
\end{align}
Note that the set $\{(j,q): j \in \cS \cup \cS^\ep\setminus
\cS^\ep_q, ~q \in \cS \}$ is the same as the set
$\{(j,q): j \in \cS \cup \cS^\ep, ~ q \in \cS \setminus \{n(j)\} \}$,
so we can rewrite \eqref{eq:Pf16} as
\begin{equation*}
  \|\bG_\cS^\star \bG (\bX^{\ep,r} - \overline{\bX^{\ep,r}})\|_{1,1} \le
  \sum_{v=1}^{N_v} \sum_{j \in \cS \cup \cS^\ep} \Big|(\bX^{\ep,r} -
  \overline{\bX^{\ep,r}})_{j,v}\Big| \sum_{q \in \cS \setminus
    \{n(j)\}}|\mu(\bg_q,\bg_j)|.
\end{equation*}
The last sum in this equation is bounded above by the interaction
coefficient $\cI_1$, and using the definition of the $\|\cdot
\|_{1,1}$ norm we get
\begin{equation}
  \label{eq:Pf17}
  \|\bG_\cS^\star \bG (\bX^{\ep,r} - \overline{\bX^{\ep,r}})\|_{1,1} \le
  \cI_1 \|\bX^{\ep,r} - \overline{\bX^{\ep,r}}\|_{1,1}.
\end{equation}

With a similar calculation we obtain
\begin{align*}
\Big\| ({\itbf I} - \bG_\cS^\star \bG_\cS) (\bxi^{\ep,r} -
  \overline{\bX^{\ep,r}}) \Big\|_{1,1} &= \sum_{q\in \cS} \sum_{v=1}^{N_v}
  \Big| \sum_{j \in \cS}(\bG_\cS^\star \bG_\cS-{\itbf I})_{q,j} (\bxi^{\ep,r} -
  \overline{\bX^{\ep,r}})_{j,v} \Big| \\
  & \le 
  \sum_{j \in \cS} \sum_{v=1}^{N_v} |(\bxi^{\ep,r} -
  \overline{\bX^{\ep,r}})_{j,v}| \sum_{q\in \cS} |\mu(\bg_q,\bg_j)-\delta_{q,j}|
  \\&= \sum_{j \in \cS} \sum_{v=1}^{N_v} |(\bxi^{\ep,r} -
  \overline{\bX^{\ep,r}})_{j,v}| \sum_{q\in \cS \setminus \{j\}} |\mu(\bg_q,\bg_j)|,
\end{align*}
where we used the triangle inequality, the identity $(\bG_\cS^\star
\bG_\cS)_{q,j} = \mu(\bg_q,\bg_j)$ and $\mu(\bg_q,\bg_q) = 1$. The
last sum is bounded above by the interaction coefficient $\cI_1$, and
using that $\bxi^{\ep,r} - \overline{\bX^{\ep,r}}$ is row supported in $\cS$,
and the definition of the $\|\cdot \|_{1,1}$ norm, we get
\begin{equation}
  \Big\| ({\itbf I} - \bG_\cS^\star \bG_\cS) (\bxi^{\ep,r} -
  \overline{\bX^{\ep,r}}) \Big\|_{1,1}  \le \cI_1 \|\bxi^{\ep,r} -
  \overline{\bX^{\ep,r}} \Big\|_{1,1}.
  \label{eq:Pf18}
\end{equation}

Gathering the results \eqref{eq:Pf14}, \eqref{eq:Pf17}--\eqref{eq:Pf18},
and using the triangle inequality, we obtain the bound
\begin{equation}
  (1- \cI_1)\|\bxi^{\ep,r} -\overline{\bX^{\ep,r}}\|_{1,1} \le \cI_1
  \|\bX^{\ep,r} - \overline{\bX^{\ep,r}}\|_{1,1} \le \cI_1
  \Big(\|\bX^{\ep,r}\|_{1,1} + \|\overline{\bX^{\ep,r}}\|_{1,1}
  \Big).\label{eq:Pf19}
\end{equation}
We also have from the definition \eqref{eq:th25} and the inequality
$|\mu(\bg_j,\bg_l)| \le 1$ for all $j, l = 1, \ldots, N_{\by}$, that
\[
\|\overline{\bX^{\ep,r}}\|_{1,1} \le  \|\bX^{\ep,r}\|_{1,1}.
\]
The estimate \eqref{eq:th26} in Theorem \ref{thm:2} follows by substituting
this in \eqref{eq:Pf19}. $\Box$

\subsection{Proof of Proposition \ref{prop:1}}
\label{sect:proof3}
Recall from section \ref{sect:proof0} the definition of the unit row vectors
$\wh{\bx}_{q \rightarrow}$. Because the rows of $\bX$ are assumed
orthogonal in the proposition, $\{\wh \bx_{q\rightarrow}, q \in \cS
\}$ is an orthonormal subset of $\CC^{1 \times N_v}$, and we conclude
from Bessel's inequality that
\begin{equation*}
  \sum_{q \in \cS \setminus\{n(j) \}} |\lin \bvr, \wh \bx_{q \rightarrow}\rin |^2
  \le \|\bvr\|_2^2, \quad \forall ~ \bvr \in \CC^{1 \times N_v} ~
\mbox{and} ~j = 1, \ldots, N_{\by}.
\end{equation*}
Dividing both sides in this equation by $\|\bvr\|_2^2$ and recalling 
definition \eqref{eq:th15}, we obtain
\begin{equation}
  \sum_{q \in \cS \setminus\{n(j) \}} |\mu(\bvr, \wh \bx_{q \rightarrow})|^2
  \le 1, \quad \forall ~ \bvr \in \CC^{1 \times N_v} ~
\mbox{and} ~j = 1, \ldots, N_{\by}.
    \label{eq:O2}
\end{equation}
For a given $j$ and $v$, we define the vector $\bnu^{(j, \bvr)} \in \RR^{1
  \times (|\cS|-1)}$ with entries $|\mu(\bvr, \wh \bx_{q
  \rightarrow})|$. Recall also from section \ref{sect:orthogX} the
vector $\bga^{(j)} \in \RR^{1 \times (|\cS|-1)}$ with entries
$|\mu(\bg_j, \bg_q)|$, for $q \in \cS \setminus\{n(j) \}$, which is a
set with cardinality $|\cS|-1$.  Using these vectors, we have
\begin{equation*}
  \sup_{\bvr \in \CC^{1 \times N_v}} \sum_{q \in \cS \setminus
    \{n(j)\}} |\mu(\bg_j,\bg_q)| |\mu(\bvr, \wh \bx_{q \rightarrow})|
  = \sup_{\bnu^{(j,\bvr)} \in \RR^{1 \times |\cS|-1}, \|\bnu^{(j,\bvr)}\|
    \le 1} \left( \bnu^{(j,\bvr)}, \bga^{(j)} \right) = \|\bga^{(j)}
  \|_2,
\end{equation*}
where $( \cdot, \cdot )$ is the Euclidian inner product in
$\RR^{1 \times |\cS|-1}$ and we used inequality \eqref{eq:O2} to
conclude that $\bnu^{(j)}$ lies in the unit ball in $\RR^{1 \times
  |\cS|-1}$. The last equality is because the sup is achieved for
$
\bnu^{(j,\bvr)} = {\bga^{(j)}/}{\|\bga^{(j)} \|_2}. $ Substituting in
the definition \eqref{eq:th13}, we obtain the result
\eqref{eq:Prop1}. $\Box$

\subsection{Proof of cluster results}
\label{sect:proof4}
The proof of Theorem \ref{thm:3} is the same as in section
\ref{sect:proof1}, with $\bX$ replaced by $\bU$, $\cS$ replaced by
$\cC$ and $\bW$ replaced by $\bcW$. This leads to the estimate 
\[
    \|\bE^{\ep,r}\|_{1,2} \le \frac{2 \cI_{N_v}^{\bU}}{r}
    \|\bX^\ep\|_{1,2} + \frac{1}{r} \big\| \big(\bG^\star \bG(\bX^\ep
    - \bX + \bR)\big)_{\cC \rightarrow} \big\|_{1,2} = \frac{2 \cI_{N_v}^{\bU}}{r}
    \|\bX^\ep\|_{1,2} + \frac{1}{r} \big\| \big(\bG^\star \bW^\ep\big)_{\cC \rightarrow} \big\|_{1,2},
\]
where we used that $\bU = \bX - \bR$, the definition of $\bW^\ep$ in Theorem \ref{thm:1} and $\bG_\cC^\star \bG \bR = 0$. 
 $\Box$

It remains to prove Lemma \ref{lem:Cl1}. The projection \eqref{eq:Cl3} that defines $\bU$  induces the  linear operator
$\mathfrak{T}: \CC^{N_r \times N_v} \to \CC^{N_r \times N_v}$ that maps $
\bG \bU = \mathfrak{T} \bG \bX. $
Note that  $ \bG \bU = \mathfrak{T} \bG \bU$ and since 
\[
0 = \bG_\cC^\star \bG \bR = \bG_\cC^\star \bG (\bX - \bU) =
\bG_\cC^\star (\bG \bX - \mathfrak{T} \bG \bX ),
\]
$\mathfrak{T}$ is the orthogonal projection onto the range of
$\bG_\cC$.
To estimate
\begin{equation}
\|\bG \bR\|_F^2 = \|\bG(\bX - \bU)\|_F^2 = \sum_{v = 1}^{N_v} \|\bG(\bx_v-\bu_v)\|_2^2 
= \sum_{v = 1}^{N_v} \|\bG \bx_v- \mathfrak{T} \bG \bx_v\|_2^2, 
\label{eq:PCl5}
\end{equation}
we note that since $\mathfrak{T}$ is the orthogonal projection
on $\mbox{range}(\bG_\cC)$,
\begin{equation}
\|\bG \bx_v- \mathfrak{T} \bG \bx_v\|_2 \le \|\bG \bx_v - \bz)\|_2, \quad 
\forall \, \bz \in \mbox{ range} (\bG_\cC).
\label{eq:deleted}
\end{equation}

Now let us define the "effective cluster matrix" $\overline{\bX}$, with entries
\begin{equation}
  \overline{\bX}_{j,v} = \left\{ \begin{array}{ll} \displaystyle
    \sum_{l \in \mathscr{S}_j} \bX_{l,v}\, \mu(\bg_j,\bg_l), \quad & j
    \in \cC, \\ 0, & \mbox{otherwise}, \end{array} \right. \quad
  \mbox{for} ~ ~1 \le j \le N_{\by}, ~ 1 \le v \le N_v.
  \label{eq:Cl2}
\end{equation}
We use the inequality \eqref{eq:deleted} for $\bz = \bG \overline{\bX} = \bG_\cC
\overline{\bX}_{\cC\rightarrow}$, and obtain
\begin{align} 
\|\bG \bx_v- \mathfrak{T} \bG \bx_v\|_2 &\le \|\bG \bx_v - \bG
\overline{\bx_v})\|_2 = \Big\| \sum_{j\in \cS} X_{j,v}\bg_j - \sum_{j
  \in \cC}\overline{X_{j,v}} \bg_j\Big\|_2, \label{eq:PCl7}
\end{align}
because $\bX$ is row supported in $\cS$ and $\overline{\bX}$ is row
supported in $\cC$.  Next, using the decomposition \eqref{eq:Cl1} of
$\cS$ and the definition \eqref{eq:Cl2} of $\overline{\bX}$, we have 
\begin{align}
\Big\| \sum_{j\in \cS} X_{j,v}\bg_j - \sum_{j \in
  \cC}\overline{X_{j,v}} \bg_j\Big\|_2 &= \Big \| \sum_{j \in \cC}
\sum_{l \in \mathscr{S}_j} X_{l,v} \bg_l - \sum_{j\in \cC}
\Big[\sum_{l\in \mathscr{S}_j} X_{l,v} \mu(\bg_j,\bg_l) \Big] \bg_j 
\Big\|_2 \nonumber \\
&= \Big \| \sum_{j \in \cC}
\sum_{l \in \mathscr{S}_j} X_{l,v} \big[\bg_l - \mu(\bg_j,\bg_l) \bg_j\big] \Big\|_2.
\end{align}
We can bound this using the triangle inequality and 
\begin{align}
\big\|\bg_l - \mu(\bg_j,\bg_l) \bg_j\big\|_2^2 = \lin \bg_l -
\mu(\bg_j,\bg_l) \bg_j, \bg_l - \mu(\bg_j,\bg_l) \bg_j\rin = 1 -
|\mu(j,l)|^2 \le 2 \cd(j,l), \label{eq:PCl8}
\end{align}
where we used the definition of the semimetric $\cd$ and of $\mu$.
Since $\mathscr{S}_j$ is contained within a ball of radius $\rc$
centered at $j \in \cC$, we have $\cd(j,l) \le \rc$ in \eqref{eq:PCl8},
and gathering the results
\eqref{eq:PCl7}--\eqref{eq:PCl8}, we get
\begin{equation}
\|\bG \bx_v- \mathfrak{T} \bG \bx_v\|_2 \le \sqrt{2 \rc} \sum_{j \in
  \cC} \sum_{l \in \mathscr{S}_j} |X_{l,v}| = \sqrt{2 \rc} \|\bx_v\|_1.
\label{eq:Pcl9}
\end{equation}
Finally, substituting in \eqref{eq:PCl5},
\begin{align*}
\|\bG \bR\|_F^2 \le  2 \rc \sum_{v = 1}^{N_v}
\big(\|\bx_v\|_1)^2 = 2 \rc \|\bX^T\|_{2,1}^2. \quad \Box
\end{align*}

\section{Summary}
\label{sect:sum}
We presented a novel resolution theory for synthetic aperture radar (SAR) imaging 
using the multiple measurement vector (MMV) approach, also known as simultaneously 
sparse optimization. This seeks to find an unknown matrix $\bX$ with sparse row support,  
by inverting a linear system of equations using sparsity promoting convex optimization.  
In the SAR imaging application, $\bX$ models
the unknown reflectivity of a scattering scene. The rows of 
$\bX$ are indexed by the points in the imaging region, and the columns correspond to its values for multiple views
of the imaging scene, from different sub-apertures and polarization diverse measurements.

The resolution theory does not pursue the question of exact recovery, but seeks  to estimate 
the neighborhood of the support of $\bX$ where the largest entries in the reconstruction lie. The radius 
of this neighborhood represents the resolution limit and it depends on the noise level. 
We introduced a quantifier of how the unknowns influence each other 
in imaging, called the multiple view interaction coefficient, and showed that the smaller this is and the weaker the noise,  the better 
the estimate of the support of $\bX$. We also quantified the error of the reconstruction  and studied the advantage of having 
multiple views.  The existing literature shows that the MMV method does not always perform better than 
sparsity promoting optimization with a single view, the so-called single measurement vector (SMV) formulation. 
We showed that if the rows of $\bX$ are orthogonal, then the MMV approach is expected to perform better, 
depending on how the unknowns are distributed in the imaging scene. We quantified this advantage and 
explained how the condition of orthogonality of the rows of $\bX$ arises in the application of SAR imaging 
of direction dependent reflectivity. 

We also studied imaging of  well-separated clusters of scatterers and showed that the MMV approach 
gives a reconstruction  supported near these clusters.

\section*{Acknowledgments}
This material is based upon work supported by the Air Force Office of
Scientific Research under award number FA9550-15-1-0118.

\appendix
\section{Proof of Proposition \ref{prop:SARort}}
\label{ap:SARort}
Let us introduce the notation
\begin{equation}
\xi_{q,v} = 
\rho_q(\obr_v,\om) \sqrt{N_r} \, \varphi \Big[ b \Big(\ot + \frac{2
    \bmm_1 \cdot \Delta \by_q}{c} \Big)\Big].
\label{eq:SA26}
\end{equation}
Assuming that $\varphi$ is smooth and using that the spacing between 
the centers of consecutive sub-apertures is small,  we approximate the correlation of the rows of $\bX$ by 
\begin{align}
\big|\mu(\bx_{q \rightarrow}, \bx_{l \rightarrow})\big| \approx \frac{
  \left|\displaystyle \int_{-A/2}^{A/2} dr \, \psi_{q,l}(r)\exp \Big[ 2 i \ok
    \frac{(\br_o + r \btau - \oby)}{|\br_o + r \btau - \oby|}
    \cdot (\by_q -\by_l) \Big]\right|}{\|\psi_{q,q}\|^{1/2}_{L_1(-A/2,A/2)}
  \|\psi_{l,l}\|^{1/2}_{L_1(-A/2,A/2)}}.
\label{eq:SA28}
\end{align}
Recall that $\br_o$ is the center of the large linear aperture, along   $\btau$.  We parametrize 
this  aperture
by the arclength $r \in [-A/2,A/2]$,  and  $\psi_{q,l}(r)$ is the smooth kernel satisfying the interpolation conditions
\begin{equation}
\psi_{q,l}\left(r = \Big( \frac{v-1}{N_v-1} - \frac{1}{2} \Big) A
\right) = \xi_{q,v} \xi_{l,v}^\star.
\label{eq:SA29}
\end{equation}
To estimate \eqref{eq:SA28}  we  expand the exponent in $r$
\begin{equation}
\ok \frac{(\br_o + r \btau - \oby)}{|\br_o + r \btau - \oby|}
\cdot (\by_q -\by_l) = \ok \bmm_o \cdot (\by_q -\by_l) + \ok r
\frac{\btau \cdot \mathbb{P}_o (\by_q - \by_l)}{|\br_o - \oby|} +
\ldots,
\label{eq:SA30}
\end{equation}
with $ \bmm_o$ and $\mathbb{P}_o$ defined as in Proposition \ref{prop:SARort}.
Suppose that  $A$ and the cross-range offset between
$\by_q$ and $\by_l$ are small enough so we can neglect the higher
terms\footnote{The results are qualitatively the same if we include
  quadratic terms in $r$ and neglect cubic and higher order terms. The
  discussion is simpler if we consider only the shown terms in
  \eqref{eq:SA30}.} in \eqref{eq:SA30}. Then, using $Q$ defined in Proposition \ref{prop:SARort}
 and integrating by parts in \eqref{eq:SA28}, we obtain 
 \begin{align}
\hspace{-0.06in}|\mu(\bx_{q\rightarrow},\bx_{l\rightarrow})| \approx
\frac{ A\left| \psi_{q,l}(A/2) e^{i Q/2} - \psi_{q,l}(- A/2) e^{-i Q/2}
  - \int_{-A/2}^{A/2} dr \, \psi'_{q,l}(r) e^{i r Q/A} \right|}{|Q|
  \|\psi_{q,q}\|^{1/2}_{L_1(-A/2,A/2)}
  \|\psi_{l,l}\|^{1/2}_{L_1(-A/2,A/2)}}.
\label{eq:SA34}
\end{align}

If the reflectivities are independent of direction, \eqref{eq:SA34}
becomes
$
|\mu(\bx_{q\rightarrow},\bx_{l\rightarrow})| \approx |\mbox{sinc}(Q/2)|.
$
This attains its maximum at $Q = 0$ i.e., at $q = l$, and decays as
$1/|Q|$, as stated in the proposition.  It remains to show that the
result extends to reflectivities that vary smoothly with direction.
We obtain from \eqref{eq:SA34}, using the triangle inequality, that
\begin{align}
|\mu(\bx_{q\rightarrow},\bx_{l\rightarrow})| \le \frac{  A \big[
    |\psi_{q,l}(A/2)| + |\psi_{q,l}(- A/2)| + \|\psi'_{q,l}
    \|_{L_1(-A/2,A/2)} \big]}{|Q|
    |\|\psi_{q,q}\|^{1/2}_{L_1(-A/2,A/2)}
    |\|\psi_{l,l}\|^{1/2}_{L_1(-A/2,A/2)}},
\label{eq:SA35}
\end{align}
and we estimate next the three terms in the numerator. 
We begin with 
\[
|\psi_{q,l}(A/2)| \le |\psi_{q,l}(s)| + \Big|\int_{s}^{A/2} d r \, \psi'_{q,l}(r)\Big|,
\]
where we used the fundamental theorem of calculus and the triangle
inequality. Therefore,
\begin{align} 
A |\psi_{q,l}(A/2)| &= \int_{-A/2}^{A/2} ds \, |\psi_{q,l}(A/2)|
\le \int_{-A/2}^{A/2} ds \, |\psi_{q,l}(s)| +
\int_{-A/2}^{A/2} ds \Big| \int_{s}^{A/2} dr \, \psi_{q,l}'(r) \Big|
\nonumber \\ &\le \|\psi_{q,l}\|_{L_1(-A/2,A/2)} + \int_{-A/2}^{A/2}
ds \int_{-A/2}^{A/2} dr \big| \psi_{q,l}'(r) \big| \nonumber \\ &=
\|\psi_{q,l}\|_{L_1(-A/2,A/2)} + A \|\psi'_{q,l}\|_{L_1(-A/2,A/2)}.
\label{eq:SA37}
\end{align}
The first term in this equation can be bound using the Cauchy-Schwartz
inequality, once we recall the definition \eqref{eq:SA29} of
$\psi_{q,l}$. We rewrite this definition as
\begin{equation}
\label{eq:defpsi}
\psi_{q,l}(r) = \xi_q(\br_o + r \btau) \xi_l^\star(\br_o + r \btau),
\end{equation}
in an abuse of notation, so that 
$
\xi_{q,v} = \xi_q(\br_v),$ for  $\br_v = 
\br_o + \Big(\frac{v-1}{N_v-1} - \frac{1}{2}\Big)A \btau.
$
We obtain that 
\begin{align*}
L_2
\| \psi_{q,l}\|_{L_1(-A/2,A/2)} &= \int_{-A/2}^{A/2} dr \, |\xi_q(\br_o + r \btau)
\xi_l^\star(\br_o + r \btau)| \nonumber \\ &\le \left[ \int_{-A/2}^{A/2} dr \,
  |\xi_q(\br_o + r \btau)|^2\right]^{1/2} \left[ \int_{-A/2}^{A/2} dr \, |\xi_l
  (\br_o + r \btau)|^2\right]^{1/2} \nonumber \\ &= \|\psi_{q,q}\|^{1/2}_{L_1(-A/2,A/2)}
\|\psi_{l,l}\|^{1/2}_{L_1(-A/2,A/2)}.
\end{align*}
We also have from \eqref{eq:defpsi} that 
\[
\psi_{q,l}'(r) = \btau \cdot \nabla \xi_q(\br_o + r \btau)
\xi_l^\star(\br_o + r \btau) + \xi_q(\br_o + r \btau) \btau
\cdot \nabla \xi_l^\star(\br_o + r \btau),
\]
and from the Cauchy-Schwartz and triangle inequalities we get 
\begin{align*}
\|\psi_{q,l}'\|_{L_1(-A/2,A/2)} &\le \|\btau \cdot \nabla \xi_q \|_{L_2(-A/2,A/2)} 
\|\xi_l \|_{L_2(-A/2,A/2)} + \|\xi_q \|_{L_2(-A/2,A/2)} 
\|\btau \cdot \nabla \xi_l \|_{L_2(-A/2,A/2)} \nonumber \\
&\le \|\nabla \xi_q \|_{L_2(-A/2,A/2)} 
\|\xi_l \|_{L_2(-A/2,A/2)} + \|\xi_q \|_{L_2(-A/2,A/2)} 
\|\nabla \xi_l \|_{L_2(-A/2,A/2)}.
\end{align*}
To estimate this further, let us introduce the constant $K_{q}$, which depends 
on the scale of variation of the reflectivity $\xi_q$, such that 
\[
\|\nabla \xi_q \|_{L_2(-A/2,A/2)}  \le \frac{K_q}{A} \|\xi_q \|_{L_2(-A/2,A/2)}.
\]
Since $ \|\xi_q \|_{L_2(-A/2,A/2)} =
\|\psi_{q,q}\|^{1/2}_{L_1(-A/2,A/2)}$ by definition
\eqref{eq:defpsi}, we obtain that 
\[
A\|\psi_{q,l}'\|_{L_1(-A/2,A/2)} \le (K_q + K_l)
\|\psi_{q,q}\|^{1/2}_{L_1(-A/2,A/2)}
\|\psi_{l,l}\|^{1/2}_{L_1(-A/2,A/2)}.
\]
The estimate \eqref{eq:SA37} becomes
\[
A |\psi_{q,l}(A/2)|  \le (1 + K_q + K_l) \|\psi_{q,q}\|^{1/2}_{L_1(-A/2,A/2)}
\|\psi_{l,l}\|^{1/2}_{L_1(-A/2,A/2)}, 
\]
and a similar bound applies to $A|\psi_{q,l}(-A/2)|$. 
Gathering the results and
substituting in \eqref{eq:SA35}, we obtain the statement of the 
proposition, with $C_{q,l} = 12 \pi (1 + K_q + K_q)$.  $\Box$

\section{Expression of matrix $\boldsymbol{\Gamma}(\obr)$}
\label{ap:Gamma}
The $6 \times 6$ matrix $\boldsymbol{\Gamma}(\obr)$ that enters the data model \eqref{eq:Pol8}
can be written as 
\begin{equation*}
\boldsymbol{\Gamma}(\obr) = \boldsymbol{\Gamma}_{\mbox{diag}}(\obr) + \boldsymbol{\Gamma}_{\mbox{off-diag}}(\obr)
\end{equation*}
where 
\begin{align*}
\boldsymbol{\Gamma}_{\mbox{diag}}(\obr) = \mbox{diagonal} \Big( (1-\eta_1^2)^2 ,(1-\eta_2)^2,(1-\beta^2)^2, (1-\eta_1^2)(1-\eta_2^2)+ (\eta_1 \eta_2)^2, \\
(1-\eta_1^2)(1-\beta^2)+(\eta_1\beta)^2, (1-\eta_2^2)(1-\beta^2) + (\eta_2 \beta)^2 \Big),
\end{align*}
is the diagonal part of $\boldsymbol{\Gamma}(\obr)$ and 
\begin{equation*}
\boldsymbol{\Gamma}_{\mbox{off-diag}}(\obr) = \begin{mpmatrix} 0 & (\eta_1 \eta_2)^2 & (\eta_1
  \beta)^2 & \eta_1 \eta_2(\eta_1^2-1) & \eta_1 \beta(\eta_1^2-1) & \eta_1^2 \eta_2 \beta \\ 
  (\eta_1 \eta_2)^2 & 0 &(\eta_2 \beta)^2& \eta_1 \eta_2(\eta_2^2-1)& \eta_1\eta_2^2 \beta& \eta_2\beta(\eta_2^2-1)
  \\ (\eta_1 \beta)^2 &(\eta_2 \beta)^2 &0 &\eta_1 \eta_2 \beta^2 &\eta_1 \beta (\beta^2-1) & \eta_2 \beta (\beta^2-1)
  \\ 2 \eta_1 \eta_2(\eta_1^2-1) & 2 \eta_1 \eta_2(\eta_2^2-1) &2 \eta_1 \eta_2 \beta^2 & 0& 
  \eta_2 \beta(2 \eta_1^2-1)& \eta_1 \beta(2 \eta_2^2-1) \\ 2 \eta_1 \beta (\eta_1^2-1) &2\eta_1 \eta_2^2 \beta 
  & 2 \eta_1 \beta (\beta^2-1) & \eta_2 \beta(2 \eta_1^2-1) & 0
  & \eta_1 \eta_2 (2 \beta^2-1) \\
  2 \eta_1^2 \eta_2 \beta & 2 \eta_2 \beta(\eta_2^2-1)& 2 \eta_2 \beta(\beta^2-1) & \eta_1 \beta(2 \eta_2^2-1) &\eta_1 \eta_2(2\beta^2-1)& 
 0
\end{mpmatrix}
\end{equation*}
is its off-diagonal part.

\bibliographystyle{siam} \bibliography{SPARSE.bib}

\end{document}